\documentclass[12pt]{amsart}
\usepackage[pdftex]{graphicx}
\usepackage{amsmath,amssymb,stmaryrd}
\usepackage{amsthm, amsfonts}
\usepackage{enumerate}
\usepackage{comment}
\usepackage[all]{xy}
\usepackage[pdftex]{color}
\usepackage{hyperref}
\usepackage[margin=27truemm]{geometry}

\theoremstyle{plain}
\newtheorem{thm}{Theorem}[section]
\theoremstyle{definition}
\newtheorem{dfn}[thm]{Definition}
\newtheorem{rem}[thm]{Remark}
\theoremstyle{plain}
\newtheorem{lem}[thm]{Lemma}
\newtheorem{prop}[thm]{Proposition}
\newtheorem{cor}[thm]{Corollary}
\newtheorem{clm}{Claim}
\theoremstyle{definition}
\newtheorem{case}{Case}

\numberwithin{equation}{section}

\graphicspath{{./fig/}}

\title{grid homology for spatial graphs and a K\"{u}nneth formula of connected sums}
\author{Hajime Kubota}
\date{}

\begin{document}
\begin{abstract}
We define the hat and tilde versions of the grid homology for spatial graphs possibly with sinks, sources, or cut edges by extending the grid homology developed by Harvey, O'Donnol \cite{Heegaard_Floer_homology_of_spatial_graphs}.
We define a cut edge for spatial graphs and show that the grid homology for a spatial graph $f$ is trivial if $f$ has a sink, source, or cut edge.
As an application, we give purely combinatorial proofs of some formulas including a K\"{u}nneth formula for the knot Floer homology of connected sums in the framework of the grid homology.
\end{abstract}
\maketitle

\section{Introduction}
Knot Floer homology is a powerful invariant of knots developed by Ozsv\'{a}th and Szab\'{o} \cite{Holomorphic-disks-and-knot-invariants} and Rasmussen \cite{1st-KFH-ras} independently.
It is a categorification of the Alexander polynomial since the graded Euler characteristic coincides with the Alexander polynomial.

Grid homology is a combinatorial reconstruction of knot Floer homology developed by Manolescu, Ozsv\'{a}th, Szab\'{o}, and Thurston \cite{oncombinatorial}.
Grid homology enables us to calculate knot Floer homology without the holomorphic theory.
So it is interesting to give a purely combinatorial proof of known results in knot Floer homology using grid homology.
For example, Sarkar \cite{grid-tau} defined the combinatorial Ozsv\'{a}th-Szab\'{o} Tau-invariant, which is defined in \cite{Knot-Floer-homology-and-the-four-ball-genus}, and gave a purely combinatorial proof that the Ozsv\'{a}th-Szab\'{o} Tau-invariant is a concordance invariant.
Similarly, F\"{o}ldv\'{a}ri \cite{grid-upsilon} reconstructed the combinatorial Upsilon invariant using grid homology. 
Then the author \cite{grid-upsilon-concordance} proved that it is a concordance invariant.

In 2017, Harvey and O'Donnol \cite{Heegaard_Floer_homology_of_spatial_graphs} extended grid homology to a certain class of oriented spatial graphs called transverse spatial graphs.
For a sinkless and sourceless transverse spatial graph $f$, they defined a sutured manifold $(E(f),\gamma(f))$ determined by $f$ and showed that the hat version of the grid homology of $f$ is isomorphic to sutured Floer homology of $(E(f),\gamma(f))$ \cite[Theorem 6.6]{Heegaard_Floer_homology_of_spatial_graphs}.
As a corollary, the graded Euler characteristic of their hat version coincides with the torsion invariant $T(E(f),\gamma(f))\in\mathbb{Z}[H_1(E(f))]$ of Friedl, Juh\'{a}sz, and Rasmussen \cite{decategorification-of-sutured-Floer}.
Bao \cite{HF-bipartite} defined Floer homology for embedded bipartite graphs.
Harvey and O'Donnol showed that Bao's Floer homology is essentially the same as their grid homology.

For sinkless and sourceless transverse spatial graphs, Harvey and O'Donnol first defined the minus version and then the hat version using it.
Both sinkless and sourceless conditions are necessary to define their minus version but not for the hat and tilde versions.

In this paper, based on their work, we quickly define the hat and tilde versions for general transverse spatial graphs.
We define a cut edge for spatial graphs (see Definition \ref{dfn:cutedge}).
We show that the grid homology for $f$ is trivial if $f$ has a sink, source, or cut edge (Theorem \ref{thm:cutedge}).
As applications of Theorem \ref{thm:cutedge}, we give some formulas (Corollary \ref{cor:wedge sum}-Theorem \ref{thm:disjoint}), including a K\"{u}nneth formula for knot Floer homology of connected sums (Corollary \ref{cor:connected-sum-knot}).

The behavior of knot Floer homology under connected sums is well-known however the connected sum operations had not been dealt with in grid homology.
This operation using grid diagrams is not written in the grid homology book \cite{grid-book}.
The difficulty of dealing with this operation can be seen from the paper of V\'{e}rtesi \cite{Transversely-nonsimple-knots}; to prove the additivity of the Legendrian and transverse invariants under connected sums, she used the identification with grid homology and the knot Floer homology due to a lack of connected sum formula in grid homology.
In particular, her proof is not combinatorial, even though these invariants are defined in grid homology.
Furthermore, the number of generators for the grid chain complexes $\widehat{GC}(K_1)\otimes\widehat{GC}(K_2)$ and $\widehat{GC}(K_1\#K_2)$ are $n_1!\times n_2!$ and $(n_1+n_2)!$ respectively and there is a natural injection from the generators of $\widehat{GC}(K_1)\otimes\widehat{GC}(K_2)$ to those of $\widehat{GC}(K_1\#K_2)$.
This suggests that most generators of $\widehat{GC}(K_1\#K_2)$ should vanish in their homology, which is not obvious from the definition.

The grid homology for spatial graphs was defined in 2007 but has few applications.
This paper gives a new application; trivial homology of spatial graphs with cut edges quickly deduces the connected sum formula.
It is more reasonable to use the grid homology for spatial graphs to show the connected sum formula than to consider the knot grid homology because spatial graphs with cut edges are represented by nice grid diagrams and we can relatively easily check that their grid homologies are trivial.

\begin{comment}
    The key idea is that by taking nice graph grid diagrams, the grid chain complexes of a connected sum or disjoint union of two spatial graphs and the grid chain complex of some spatial graph with a cut edge can be almost the same.
Let $g_1$ and $g_2$ be two graph grid diagrams for spatial graphs $f_1$ and $f_2$ respectively.
We can take the induced graph grid diagram $g$ for a connected sum or disjoint union of $f_1$ and $f_2$.
Then there exists a graph grid diagram $g'$ representing some spatial graph with a cut edge such that the grid chain complexes of $g$ and $g'$ are written simultaneously as $\widehat{CF}(g)=\mathrm{Cone}(\phi\colon A\to B)$ and $\widehat{CF}(g')=\mathrm{Cone}(\phi'\colon A\to B)$ respectively and that $A$, $B$, and $\widehat{CF}(g_1)\otimes\widehat{CF}(g_2)$ are quasi-isomorphic.
Then there exists a graph grid diagram $g'$ representing some spatial graph with a cut edge such that
\begin{itemize}
    \item the grid chain complex of $g$ is written as $\widehat{CF}(g)=\mathrm{Cone}(\phi\colon I\to N)$,
    \item the grid chain complex of $g'$ is written as $\widehat{CF}(g')=\mathrm{Cone}(\phi'\colon I\to N)$, and
    \item $I$, $N$, and $\widehat{CF}(g_1)\otimes\widehat{CF}(g_2)$ are quasi-isomorphic,
\end{itemize}
where $I$ and $N$ are some chain complexes.
To explore $I$ and $N$, we decompose them into finitely many copies of two special acyclic chain complexes (Sections \ref{sec:specialcpx}).
\end{comment}

\subsection{MOY graphs}
An oriented spatial graph $f$ is the image of an embedding of a directed graph in $S^3$.
Intuitively, a \textbf{transverse spatial graph} is an oriented spatial graph such that for each vertex, there is a small disk that separates the incoming edges and the outgoing edges.
See \cite[Definition 2.2]{Heegaard_Floer_homology_of_spatial_graphs} for the definition of transverse spatial graphs.

Let $E(f)$ denote the set of edges of $f$ and $V(f)$ the set of vertices of $f$.
For $v\in V(f)$, let $\mathrm{In}(v)$ be the set of edges incoming to $v$ and $\mathrm{Out}(v)$ be the set of edges outgoing to $v$.
\begin{dfn}
\label{dfn:balanced}
\begin{enumerate}
    \item A \textbf{balanced coloring} $\omega$ for $f$ is a map $E(f)\to\mathbb{Z}$ satisfying $\sum_{e\in \mathrm{In}(v)}\omega(e)=\sum_{e\in \mathrm{Out}(v)}\omega(e)$ for each $v\in V(f)$.
    \item An \textbf{MOY graph} $(f,\omega)$ is a pair of a transverse spatial graph $f$ and a balanced coloring of $f$.
\end{enumerate}
\begin{comment}
    We sometimes abbreviate $(f,\omega)$ as $f$ when no confusion arises.
\end{comment}
\end{dfn}

\begin{comment}
\begin{rem}
Mellor, Kong, Lewald, and Pigrish \cite{coloring} define a balanced spatial graph as a pair of a spatial graph and a balanced coloring, which is an MOY graph without the transverse condition.
It is confusing that Vance\cite{grid-tau} and the author paper\cite{grid-upsilon-concordance} defined a balanced spatial graph, which is a transverse spatial graph satisfying that the number of incoming edges equals the number of outgoing edges at each vertex.
Note that a balanced spatial graph of Vance and the author is a special case of an MOY graph.
\end{rem}
\end{comment}

We call a vertex $v$ \textbf{sink} if $v$ has only incoming edges.
Analogously, we call a vertex $v$ \textbf{source} if $v$ has only outgoing edges.

\begin{dfn}
\label{dfn:cutedge}
Let $f$ be a spatial graph.
An edge $e\in E(f)$ is a \textbf{cut edge} if there exists an embedded 2-sphere $\Sigma\subset S^3$ such that $\Sigma\cap(f(G)-\mathrm{Int}(e))=\emptyset$ and $e$ meets $\Sigma$ transversely in a single point.
\end{dfn}
\begin{rem}
\begin{itemize}
    \item The 2-sphere $\Sigma$ in the above definition is a \textit{cutting sphere} introduced by Taniyama \cite{cutting-sphere}.
    \item Cut edges for abstract graphs differ from those for spatial graphs.
    An edge $e$ of an abstract graph is called a cut edge if $G-e$ has one more connected component than $G$.
    \item Some spatial graph $f\colon G\to S^3$ has no cut edge even if $G$ has cut edges as an abstract graph. (Figure \ref{fig:3handcuff}).
\end{itemize}
\end{rem}

\begin{figure}
\centering
\includegraphics[scale=0.24]{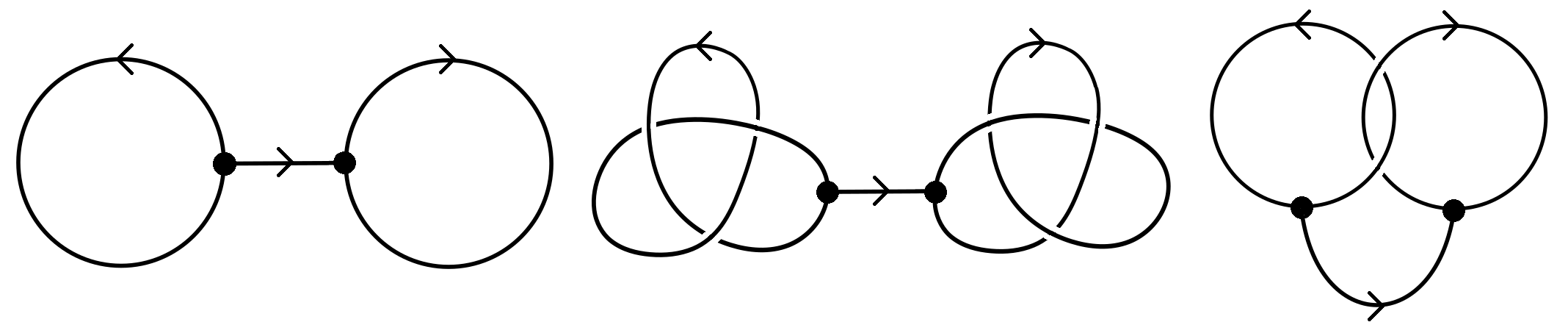}
\caption{Spatial handcuff graphs. The rightmost one has no cut edge as a spatial graph.}
\label{fig:3handcuff}
\end{figure}

\subsection{Main results}

\begin{thm}
\label{thm:cutedge}
Let $(f,\omega)$ be an MOY graph.
\begin{enumerate}
    \item If $f$ has a sink or source, then $\widehat{HF}(f,\omega)=0$.
    \item If $f$ has a cut edge as a spatial graph, then $\widehat{HF}(f,\omega)=0$.
\end{enumerate}
\end{thm}
This theorem is a kind of combinatorial version of the triviality of the sutured Floer homology for a non-taut sutured manifold \cite[Proposition 9.18]{Holomorphic-discs-and-sutured-manifolds}.

If $f$ has a sink or source $v$, then The corresponding sutured manifold $(E(f), \gamma(f))$ has a suture $s$ associated with $v$ as the boundary of the transverse disk at $v$.
Then $R_+(\gamma)$ or $R_-(\gamma)$ will contain a disjoint disk and hence $(E(f), \gamma(f))$ is not taut.
If $f$ has a cut edge as a spatial graph, then the suture associated with the cut edge will be trivial, and $(E(f), \gamma(f))$ is not taut.

We define the disjoint union, the connected sum, and the wedge sum of two MOY graphs as follows:
\begin{dfn}
\label{dfn: disjoint-connected-wedge-sum}
Suppose $(f_1,\omega_1)$ and $(f_2,\omega_2)$ are two MOY graphs and $(v_1,v_2)\in V(f_1)\times V(f_2)$.
\begin{enumerate}[(1)]
\item Let $(f_1\sqcup f_2,\omega_1\sqcup\omega_2)$ be an MOY graph as a disjoint union of $(f_1,\omega_1)$ and $(f_2,\omega_2)$, where $\omega_1\sqcup\omega_2$ is naturally determined by $\omega_1,\omega_2$.
\item If $\omega_1(v_1)=\omega_2(v_2)$, let $(f_1\#_{(v_1,v_2)}f_2,\omega_1\#\omega_2)$ be an MOY graph obtained from $(f_1\sqcup f_2,\omega_1\sqcup\omega_2)$ as in Figure \ref{fig:3MOY}, where $\omega_1\#\omega_2$ is naturally determined by $\omega_1,\omega_2$.
\item Let $(f\vee_{(v_1,v_2)} f_2,\omega_1\vee\omega_2)$ be an MOY graph obtained from $f_1\sqcup f_2$ by identifying $v_1$ and $v_2$, where $\omega_1\vee\omega_2$ is naturally determined by $\omega_1,\omega_2$.
\end{enumerate}
\begin{comment}
  These MOY graphs are sometimes abbreviated as $(f_1\sqcup f_2)$, $(f_1\#_{(v_1,v_2)}f_2)$, and $(f\vee_{(v_1,v_2)} f_2)$.  
\end{comment}
\begin{figure}
\centering
\includegraphics[scale=0.4]{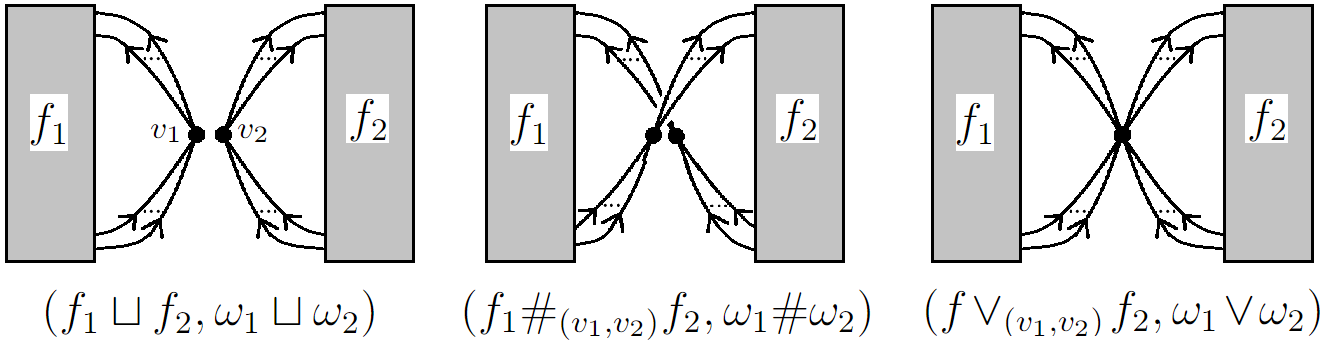}
\caption{MOY graphs in Definition \ref{dfn: disjoint-connected-wedge-sum}}
\label{fig:3MOY}
\end{figure}
\end{dfn}

\begin{cor}
\label{cor:wedge sum}
Let $(f_1,\omega_1)$ and $(f_2,\omega_2)$ be two MOY graphs.
For any pair $(v_1,v_2)\in V(f_1)\times V(f_2)$, we have
\begin{align*}
\widehat{HF}(f\vee_{(v_1,v_2)} f_2,\omega_1\vee\omega_2)\cong\widehat{HF}(f_1,\omega_1)\otimes\widehat{HF}(f_2,\omega_2).
\end{align*}
as absolute Maslov graded, relative Alexander graded $\mathbb{F}$-vector spaces.
\end{cor}

Let $W(i)$ be a two-dimensional graded vector space $W(i)\cong\mathbb{F}_{0,0}\oplus\mathbb{F}_{-1,-i}$, where $\mathbb{F}=\mathbb{Z}/2\mathbb{Z}$.
For a bigraded $\mathbb{F}$-vector space $X$, the corresponding \textbf{shift} of $X$, denoted $X\llbracket a,b\rrbracket$, is the bigraded $\mathbb{F}$-vector space so that $X\llbracket a,b\rrbracket_{d,s}=X_{d+a,s+b}$.
Then, we have 
\[
X\otimes W(i)\cong X\oplus X\llbracket1,i\rrbracket.
\]

\begin{thm}
\label{thm:connected-sum}
Let $(f_1,\omega_1)$ and $(f_2,\omega_2)$ be two MOY graphs.
Let $(v_1,v_2)\in V(f_1)\times V(f_2)$ be a pair of vertices with $\omega_1(v_1)=\omega_2(v_2)$.
Then we have
\[
\widehat{HF}(f_1\#_{(v_1,v_2)}f_2,\omega_1\#\omega_2) \cong\widehat{HF}(f_1,\omega_1)\otimes\widehat{HF}(f_2,\omega_2)\otimes W(\omega_1(v_1)),
\]
as absolute Maslov graded, relative Alexander graded $\mathbb{F}$-vector spaces.
\end{thm}

As a corollary of Theorem \ref{thm:connected-sum}, we give a purely combinatorial proof of a K\"{u}nneth formula for the knot Floer homology of connected sums.

\begin{cor}
\label{cor:connected-sum-knot}
Let $L_1$ and $L_2$ be two links.
Let $L_1\# L_2$ be the link obtained from the disjoint union of $L_1$ and $L_2$, via a connected sum of $K_1\in L_1$ and $K_2\in L_2$.
Then we have
\begin{align*}
\widehat{HFK}(L_1\# L_2)\cong\widehat{HFK}(L_1)\otimes\widehat{HFK}(L_2),
\end{align*}
as bigraded $\mathbb{F}$-vector spaces.
\end{cor}
This corollary is the combinatorial version of \cite[Theorem 1.4]{Holomorphic-disks-link-invariants-and-the-multi-variable-Alexander-polynomial}.

\begin{thm}
\label{thm:disjoint}
Let $(f_1,\omega_1)$ and $(f_2,\omega_2)$ be two MOY graphs.
Then we have
\[
\widehat{HF}(f_1\sqcup f_2,\omega_1\sqcup\omega_2) \cong\widehat{HF}(f_1,\omega_1)\otimes\widehat{HF}(f_2,\omega_2)\otimes W(0),
 \]
as absolute Maslov graded, relative Alexander graded $\mathbb{F}$-vector spaces.
\end{thm}
This theorem is the combinatorial version of \cite[Proposition 9.15]{Holomorphic-discs-and-sutured-manifolds}.

\subsection{Outline of the paper.}
In Section 2, we quickly review the grid homology for MOY graphs.
In Section 3, we give the proof of Theorem \ref{thm:cutedge}.
In Section 4, we define some acyclic chain complexes used for the proof of Theorem \ref{thm:connected-sum}.
In Section 5, we prove Theorem \ref{thm:connected-sum}.
In Sections 6-8, we verify Corollaries \ref{cor:wedge sum} and \ref{cor:connected-sum-knot}, and Theorem \ref{thm:disjoint} respectively.
Finally, in Section 9, we give an application of Theorem \ref{thm:cutedge} and some examples.

\section{grid homology for general MOY graphs}
\subsection{The definition of the grid chain complex}
\label{sec:dfn-of-grid-homology}
\begin{comment}
\subsection{Background in bigraded chain complexes}
\begin{dfn}
Let $\mathbb{F}=\mathbb{Z}/2\mathbb{Z}$.
\begin{itemize}
    \item A \textbf{bigraded $\mathbb{F}$-vector space} $X$ is an $\mathbb{F}$-vector space with a splitting $X=\bigoplus_{d,s\in \mathbb{Z}}X_{d,s}$.
    \item A \textbf{bigraded $\mathbb{F}[U_1,\dots,U_n]$-module homomorphism} $f\colon M\rightarrow M'$ is a homomorphism between two bigraded $\mathbb{F}[U_1,\dots,U_n]$-modules which sends $M_{d,s}$ to $M'_{d,s}$ for all $d,s\in\mathbb{Z}$.
    \item $\mathbb{F}[U_1,\dots,U_n]$-module homomorphism $f\colon M\rightarrow M'$ is a \textbf{homogeneous of degree $(m,t)$} if it sends $M_{d,s}$ to $M'_{d+m,s+t}$ for all $d,s\in\mathbb{Z}$.
    \item $(C,\partial)$ is a \textbf{bigraded chain complex over} $\mathbb{F}[U_1,\dots,U_n]$ if $C$ is bigraded $\mathbb{F}[U_1,\dots,U_n]$-module and if $\mathbb{F}[U_1,\dots,U_n]$-module homomorphism $\partial\colon C\rightarrow C$ satisfies $\partial\circ\partial=0$ and is homogeneous of degree $(-1,0)$.
\end{itemize}
\end{dfn}
Throughout this paper, we often consider chain maps and chain homotopy equivalences as bigraded $\mathbb{F}[U_1,\dots, U_n]$-module homomorphisms.
\end{comment}

This section provides an overview of the grid homology for MOY graphs.
It can be defined immediately from the grid homology for transverse spatial graphs by modifying its Alexander grading.
For the grid homology for sinkless and sourceless transverse spatial graphs, see \cite{Heegaard_Floer_homology_of_spatial_graphs}.

Harvey and O'Donnol \cite{Heegaard_Floer_homology_of_spatial_graphs} defined the minus version of the grid homology for transverse spatial graphs and then the hat and tilde versions.
The sinkless and sourceless condition is necessary for their minus version but is not necessary for the tilde and hat versions.
In fact, the minus version requires this condition to ensure that $\partial^-\circ\partial^-=0$.
Referring to \cite[Remark 4.6.13]{grid-book}, we will introduce the tilde and hat versions without using the minus version.

A \textbf{planar graph grid diagram} $g$ is an $n\times n$ grid of squares some of which are decorated with an $X$- or $O$- (sometimes $O^*$-) marking with the following conditions.
\begin{enumerate}[(i)]
\item There is exactly one $O$ or $O^*$ on each row and column.
\item If a row or column has no $X$ or more than one $X$, then the row or column has $O^*$.
\item $O$'s (or $O^*$'s) and $X$'s do not share the same square.
\end{enumerate}
We denote the set of $O$- and $O^*$-markings by $\mathbb{O}$, the set of $O^*$-markings by $\mathbb{O}^*$, and the set of $X$-markings by $\mathbb{X}$.
We will use the labeling of markings as $\{O_i\}_{i=1}^n$ and $\{X_j\}_{j=1}^m$.
We assume that $O_1,\dots,O_V$ are the $O^*$-markings.

\begin{rem}
In this paper, we allow grid diagrams to have "an isolated $O^*$": an $O^*$-marking with no $X$ in its row and column.
In this case, an isolated $O^*$-marking represents an isolated vertex. 
\end{rem}

A graph grid diagram realizes a transverse spatial graph by drawing horizontal segments from the $O-$ (or $O{}^*-$) markings to the $X$-markings in each row and vertical ones from the $X$-markings to the $O-$ (or $O{}^*-$) markings in each column, and assuming that the vertical segments always cross above the horizontal ones.
$O^*$-markings correspond to vertices of the transverse spatial graph and $O$- and $X$-markings to the interior of edges of a transverse spatial graph.

Throughout the paper, we only consider graph grid diagrams representing MOY graphs: any $O$-marking is connected to some $O^*$-marking by segments.

\begin{dfn}
For a graph grid diagram $g$ representing $f$ with balanced coloring $\omega$, a \textbf{weight} $\omega_g\colon\mathbb{O}\cup\mathbb{X}\to\mathbb{Z}$ is a map naturally determined by $\omega$ as follows;
\begin{itemize}
    \item $\omega_g(O_i)=\omega(e)$ if $O_i$ corresponds to the interior of the edge $e$.
    \item $\omega_g(X_j)=\omega(e)$ if $X_j$ corresponds to the interior of the edge $e$.
    \item $\omega_g(O_i)=\sum_{e\in\mathrm{In}(v)}\omega(e)=\sum_{e\in\mathrm{Out}(v)}\omega(e)$ if $O_i$ is decorated by $*$ and corresponds to the vertex $v$.
\end{itemize}
We abbreviate $\omega_g$ to $\omega$ as long as there is no confusion.
We remark that if $O_i$ represents a sink or source, then $\omega(O_i)=0$.
\end{dfn}

We regard a graph grid diagram as a diagram of the torus obtained by identifying edges in a natural way.
This is called a \textbf{toroidal graph grid diagram}.
We assume that every toroidal diagram is oriented naturally.
We write the horizontal circles and vertical circles which separate the torus into $n\times n$ squares as $\boldsymbol{\alpha}=\{\alpha_i\}_{i=1}^n$ and $\boldsymbol{\beta}=\{\beta_j\}_{j=1}^n$ respectively.

Any two graph grid diagrams representing the same transverse spatial graph are connected by a finite sequence of the graph grid moves \cite[Theorem 3.6]{Heegaard_Floer_homology_of_spatial_graphs}.
The graph grid moves are the following three moves (refer to \cite{Heegaard_Floer_homology_of_spatial_graphs})$\colon$
\begin{itemize}
\item \textbf{Cyclic permutation} (the left of Figure \ref{fig:cyccomm}) permuting the rows or columns cyclically.
\item \textbf{Commutation$'$} (the right of Figure \ref{fig:cyccomm}) permuting two adjacent columns satisfying the following condition; there are vertical line segments $\textrm{LS}_1,\textrm{LS}_2$ on the torus such that (1) $\mathrm{LS}_1\cup\mathrm{LS}_2$ contains all the $X$'s and $O$'s in the two columns, (2) the projection of $\mathrm{LS}_1\cup\mathrm{LS}_2$ to a single vertical circle $\beta_i$ is $\beta_i$, and (3) the projection of their endpoints $\partial(\mathrm{LS}_1)\cup\partial(\mathrm{LS}_2)$ to a single circle $\beta_i$ is precisely two points. Permuting two rows is defined in the same way.
\item \textbf{(De-)stabilization$'$} (Figure \ref{fig:sta'}) let $g$ be an $n\times n$ graph grid diagram and choose an $X$-marking. Then $g'$ is called a stabilization$'$ of $g$ if it is an $(n+1)\times(n+1)$ graph grid diagram obtained by adding a new row and column next to the $X$-marking of $g$, moving the $X$-marking to next column, and putting new one $O$-marking just above the $X$-marking and one $X$-marking just upper left of the $X$-marking. The inverse of stabilization is called destabilization.
\end{itemize}
These moves are also valid for MOY graphs.
Reidemeister moves around sinks and sources can be realized by these moves in the same way as the general vertices.

\begin{figure}
\centering
\includegraphics[scale=0.5]{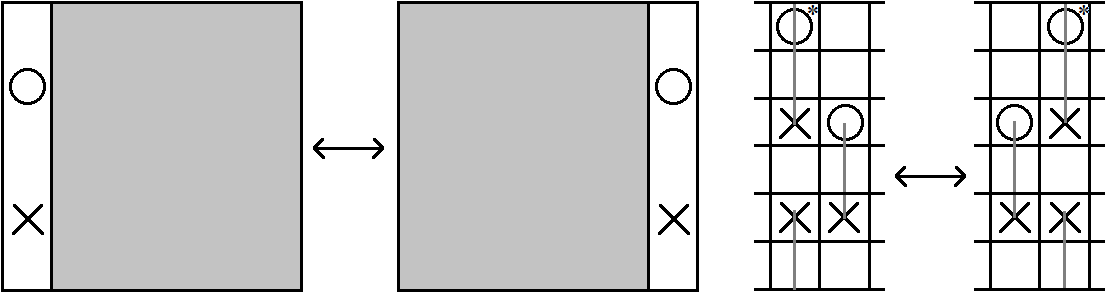}
\caption{Cyclic permutation and commutation$'$, gray lines are $\mathrm{LS}_1$ and $\mathrm{LS}_2$}
\label{fig:cyccomm}
\includegraphics[scale=0.5]{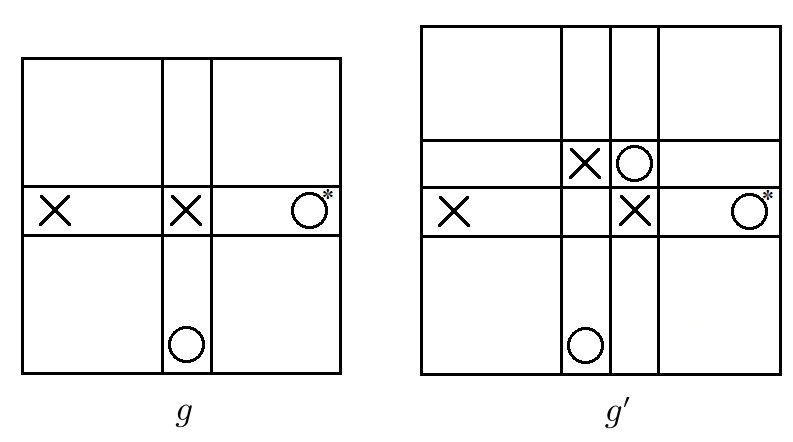}
\caption{stabilization$'$}
\label{fig:sta'}
\includegraphics[scale=0.55]{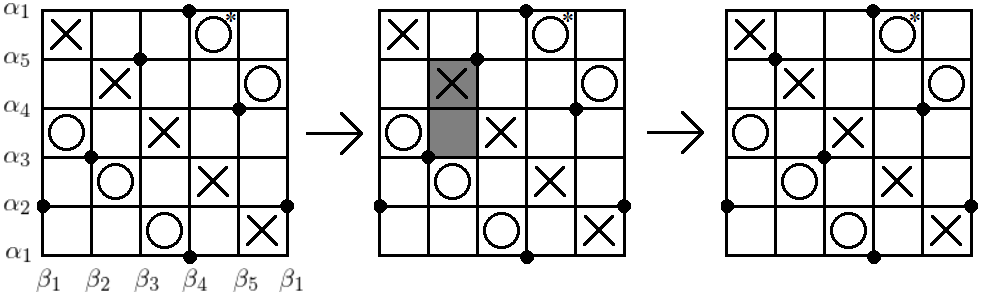}
\caption{An example of a state and a rectangle}
\label{fig:state}
\end{figure}

A \textbf{state} $\mathbf{x}$ of $g$ is a bijection $\boldsymbol{\alpha}\rightarrow\boldsymbol{\beta}$,
in other words, an $n$-tuple of points in the torus such that each horizontal circle has exactly one point of $\mathbf{x}$ and each vertical circle has exactly one point of $\mathbf{x}$.
We denote by $\mathbf{S}(g)$ the set of states of $g$.
We describe a state as $n$ points on the graph grid diagram (Figure \ref{fig:state}).

For $\mathbf{x,y}\in \mathbf{S}(g)$, a \textbf{domain} $p$ from $\mathbf{x}$ to $\mathbf{y}$ is a formal sum of the closure of squares, which is satisfying the following conditions;
\begin{itemize}
    \item $p$ is divided by $\boldsymbol{\alpha}\cup\boldsymbol{\beta}$
    \item $\partial(\partial_\alpha p)=\mathbf{y}-\mathbf{x}$ and $\partial(\partial_\beta p)=\mathbf{x}-\mathbf{y}$, where $\partial_\alpha p$ is the portion of the boundary of $p$ in the horizontal circles $\alpha_1\cup\dots\cup\alpha_n$ and $\partial_\beta p$ is the portion of the boundary of $p$ in the vertical ones.
\end{itemize}
A domain $p$ is \textbf{positive} if the coefficient of any square is nonnegative.
Here, we always consider positive domains.
Let $\pi(\mathbf{x,y})$ denote the set of positive domains from $\mathbf{x}$ to $\mathbf{y}$.

Let $\mathbf{x,y\in S}(g)$ be two states with $|\mathbf{x\cap y}|=n-2$.
An \textbf{rectangle} $r$ from $\mathbf{x}$ to $\mathbf{y}$ is a domain such that $\partial (r)$ is the union of four segments.
A rectangle $r$ is \textbf{empty} if $\mathbf{x}\cap\mathrm{Int}(r)=\mathbf{y}\cap\mathrm{Int}(r)=\emptyset$.
Let $\mathrm{Rect}^\circ(\mathbf{x,y})$ be the set of empty rectangles from $\mathbf{x}$ to $\mathbf{y}$.
If $|\mathbf{x\cap y}|\neq n-2$, then we define $\mathrm{Rect}^\circ(\mathbf{x,y})=\emptyset$.

For two domains $p_1\in\pi(\mathbf{x,y})$ and $p_2\in\pi(\mathbf{y,z})$, the \textbf{composite domain} $p_1*p_2$ is the domain from $\mathbf{x}$ to $\mathbf{z}$ such that the coefficient of each square is the sum of the coefficient of the square of $p_1$ and $p_2$.

\begin{dfn}
Let $\widetilde{CF}(g,\omega)$ be an $\mathbb{F}$-vector space finitely generated by $\mathbf{S}(g)$, with the endmorphism
\[
\widetilde{\partial}(\mathbf{x})=\sum_{\mathbf{y}\in\mathbf{S}(g)}\#\{r\in\mathrm{Rect}^\circ(\mathbf{x,y})|r\cap\mathbb{O}=r\cap\mathbb{X}=\emptyset\}\cdot\mathbf{y},
\]
where $\#\{\cdot\}$ counts rectangles modulo $2$.
\end{dfn}

\begin{dfn}
Let $\widehat{CF}(g,\omega)$ be an $\mathbb{F}$-vector space with basis $\{U_{V+1}^{k_{V+1}}\cdots U_n^{k_n}\cdot\mathbf{x}|k_i\geq0,\mathbf{x\in S}(g)\}$ with the endmorphism defined as
\[
\widehat{\partial}(\mathbf{x})=\sum_{\mathbf{y}\in\mathbf{S}(g)}\left(
\sum_{\{r\in \mathrm{Rect}^\circ(\mathbf{x,y})|r\cap\mathbb{X}=r\cap\mathbb{O}^*=\emptyset\}}
U_{V+1}^{O_{V+1}(r)}\cdots U_n^{O_n(r)}
\right)\mathbf{y}.
\]
\end{dfn}

A \textbf{planar realization} of a toroidal diagram $g$ is a planar figure obtained by cutting it along some $\alpha_i$ and $\beta_j$ and putting it on $[0,n)\times[0,n)\in\mathbb{R}^2$ in a natural way.

For two points $(a_1,a_2),(b_1,b_2)\subset\mathbb{R}^2$, we say $(a_1,a_2)<(b_1,b_2)$ if $a_1<b_1$ and $a_2<b_2$.
For two sets of finitely points $A,B\subset\mathbb{R}^2$, let $\mathcal{I}(A,B)$ be the number of pairs $a\in A,b\in B$ with $a<b$ and let $\mathcal{J}(A,B)=(\mathcal{I}(A,B)+\mathcal{I}(B,A))/2$.

We consider that $n$ points of states are on lattice point on $\mathbb{R}^2$ and each $O$- and $X$-marking is located at $(l+\frac{1}{2},m+\frac{1}{2})$ for some $l,m\in\{0,1,\dots,n-1\}$.
\begin{dfn}
\label{def-ma}
Let $\omega$ be a weight of $g$.
Take a planar realization of $g$.
For $\mathbf{x\in S}(g)$, the Maslov grading $M(\mathbf{x})$ and the Alexander grading $A(\mathbf{x})$ are defined by
\begin{align}
M(\mathbf{x})&=\mathcal{J}(\mathbf{x}-\mathbb{O},\mathbf{x}-\mathbb{O})+1,
\label{eq:Maslov}
\\
A(\mathbf{x})&=\mathcal{J}(\mathbf{x},\sum_{j=1}^m\omega(X_j)\cdot X_j-\sum_{i=1}^n\omega(O_i)\cdot O_i).
\label{eq:Alexander}
\end{align}
These two gradings are extended to the whole of $\widehat{CF}(g,\omega)$ by
\begin{align}
\label{uu}
M(U_i)=-2,\ A(U_i)=-\omega(O_i)\ (i=V+1,\dots,n).
\end{align}
\end{dfn}

The Maslov grading is well-defined as a toroidal diagram \cite[Lemma 2.4]{oncombinatorial}.
The Alexander grading is not well-defined as a toroidal diagram, however, relative Alexander grading $A^{rel}(\mathbf{x,y})=A(\mathbf{x})-A(\mathbf{y})$ is well-defined,\cite[Corollary 4.14]{Heegaard_Floer_homology_of_spatial_graphs}

\begin{prop}
$\widetilde{CF}(g,\omega),\widehat{CF}(g,\omega)$ are an absolute Maslov graded, relative Alexander graded chain complex.
We will denote by $\widetilde{HF}(g,\omega),\widehat{HF}(g,\omega)$ their homology respectively.
\end{prop}
\begin{proof}
By using the same argument such as \cite[Proposition 4.18]{Heegaard_Floer_homology_of_spatial_graphs}, the differential $\widetilde{\partial}$ and $\widehat{\partial}$ drops the Maslov grading by one and preserves the Alexander grading.

We will use the notations of \cite[Lemma 4.6.7]{grid-book} to show that $\widetilde{\partial}^2=0$ and $\widehat{\partial}^2=0$.
The cases (R-1) and (R-2) can be shown in the same way.
The case (R-3) is slightly different.
When $\mathbf{x=z}$, the composite domain of two empty rectangles is a thin annulus because the rectangles are empty.
Since every row and column has an (isolated) $O^*$-marking or at least one $X$-marking, we can not take such a domain in the hat and tilde versions.
\end{proof}

\subsection{The invariance of \texorpdfstring{$\widehat{HF}$}{HF}}
The invariance of our hat version follows immediately from the invariance of the hat version of Harvey and O'Donnol \cite{Heegaard_Floer_homology_of_spatial_graphs} because our definitions except for the Alexander grading are the same as theirs.
To prove the invariance, it is sufficient to recall the chain maps Harvey and O'Donnol gave and to take the induced maps.
\begin{comment}
The grid chain complex of their minus version is \textit{an absolute $\mathbb{Z}$-valued Maslov graded,  relative $H_1(S^3-f(G))$-valued Alexander graded $\mathbf{F}[U_1,\dots,U_V]$-module} generated by $\mathbf{S}(g)$, where $f\colon G\to S^3$ is a transverse spatial graph \cite{Heegaard_Floer_homology_of_spatial_graphs}.
They proved the invariance of 

Let $\pi(\mathbf{x,y})$ be a set of certain domains from $\mathbf{x}$ to $\mathbf{y}$.
If an $\mathbf{F}[U_1,\dots,U_V]$-module homomorphism $\phi$ of the minus version is defined as counting domains by
\[
\phi(\mathbf{x})=\sum_{\mathbf{y\in S}(g)}\left(\sum_{\{p\in\pi(\mathbf{x,y})|p\cap\mathbb{X}=\emptyset\}}
U_{1}^{O_{1}(r)}\cdots U_n^{O_n(r)}
\right)\mathbf{y},
\]
then the induced map of their hat version is given by
\[
\widehat{\phi}(\mathbf{x})=\sum_{\mathbf{y\in S}(g)}\left(\sum_{\{p\in\pi(\mathbf{x,y})|p\cap\mathbb{X}=p\cap\mathbb{O^*}=\emptyset\}}
U_{V+1}^{O_{V+1}(r)}\cdots U_n^{O_n(r)}
\right)\mathbf{y},
\]
because their hat version is obtained by letting $U_1=\dots=U_V=0$.
Define the induced map on our hat version $\widehat{\phi'}$ with
\[
\widehat{\phi'}(\mathbf{x})=\sum_{\mathbf{y\in S}(g)}\left(\sum_{\{p\in\pi(\mathbf{x,y})|p\cap\mathbb{X}=p\cap\mathbb{O^*}=\emptyset\}}
U_{V+1}^{O_{V+1}(r)}\cdots U_n^{O_n(r)}
\right)\mathbf{y}.
\]
If $\phi$ satisfies some property (grading, chain map, chain homotopy equivalence, etc.), then $\widehat{\phi'}$ does so because the argument of counting domains holds as well.
\end{comment}
So the following propositions are shown immediately.
\begin{prop}
Let $g$ and $g'$ be two graph grid diagrams for an MOY graph $(f,\omega)$.
Let $\omega_g$ and $\omega_{g'}$ be weights for $g$ and $g'$ respectively determined by $\omega$.
Then there is an isomorphism of absolute Maslov graded, relative Alexander graded $\mathbb{F}$-vector spaces
\[
\widehat{HF}(g,\omega_g)\cong\widehat{HF}(g',\omega_{g'}).
\]
We will denote by $\widehat{HF}(f,\omega)$.
\end{prop}
\begin{proof}
Let $g$ and $g'$ be two graph grid diagrams for $f$.
Suppose that
By \cite[Theorem 3.6]{Heegaard_Floer_homology_of_spatial_graphs}, it is sufficient to check the case that $g'$ is obtained by a single graph grid move.

If $g'$ is obtained from $g$ by a single cyclic permutation, then the natural correspondence of states induces the isomorphism of chain complexes $\widehat{CF}(g,\omega_g)\cong \widehat{CF}(g',\omega_{g'})$.

If $g'$ is obtained from $g$ by a single commutation$'$ or stabilization$'$, then the quasi-isomorphism of \cite[Proposition 5.1 or 5.5]{Heegaard_Floer_homology_of_spatial_graphs} induces the quasi-isomorphism of our hat chain complexes.
Let $\phi$ be the quasi-isomorphism of chain complexes of $\mathbb{F}[U_1,\dots,U_V]$-module for their minus version.
We can take the induced map $\widehat{\phi}$ of chain complexes of $\mathbb{F}$-vector space for their hat version by letting $U_1=\dots=U_V$.
Since the definitions of their hat version and our hat version are the same except for the Alexander gradings, $\widehat{\phi}$ works also for our hat chain complex as quasi-isomorphism.
We remark that sometimes $g$ has isolated $O^*$'s for our hat version, but the argument of counting rectangles works similarly.
\end{proof}

\begin{prop}
\label{prop:tilde-hat}
Let $g$ be an $n\times n$ graph grid diagram.
Then there is an isomorphism as absolute Maslov graded, relative Alexander graded $\mathbb{F}$-vector spaces
\begin{equation}
\label{eq:hat=tilde}
\widetilde{HF}(g,\omega_g)\cong\widehat{HF}(g,\omega_g)\bigotimes_{i=V+1}^n W(\omega_g(O_i)).
\end{equation}
\end{prop}
\begin{proof}
It is shown by the same arguments as \cite[Proposition 4.21, Lemma 4.31, and Proposition 4.32]{Heegaard_Floer_homology_of_spatial_graphs}.
\end{proof}

\section{The proof of Theorem \ref{thm:cutedge} (1)}
We will check that the grid chain complex for an MOY graph with sinks or sources can be written as a mapping cone $\mathrm{Cone}(\partial_N^I\colon\widetilde{N}\to\widetilde{I})$ such that the chain map $\partial_N^I$ is a quasi-isomorphism.
Then Theorem \ref{thm:cutedge} (1) follows by the standard argument of homological algebra.

\begin{proof}[\textbf{proof of theorem \ref{thm:cutedge} (1)}]
We will show the case of an MOY graph with a source.
The case of a sink can be proved in the same way by reflecting the graph grid diagram along the diagonal line.

Let $(f,\omega)$ be an MOY graph with a source.
Let $f'$ be the spatial graph obtained from $f$ by removing the source and all edges outgoing the source.
Take a balanced coloring $\omega'$ for $f'$ naturally determined by $\omega$.
\begin{figure}
\centering
\includegraphics[scale=0.45]{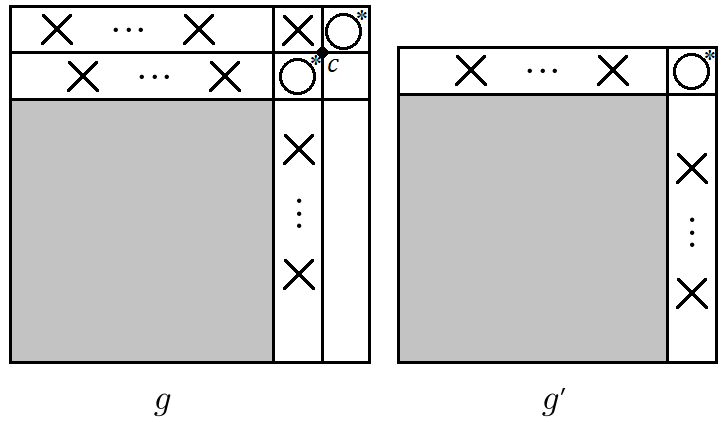}
\caption{Two graph grid diagrams representing $f$ and $f'$ respectively}
\label{fig:sink}
\end{figure}

Choose one vertex $v$ of $f$ adjacent to the source.
Take two graph grid diagrams $g$ and $g'$ for $f$ and $f'$ respectively.
Suppose that $g'$ is obtained from $g$ by removing the top row and rightmost column of $g$ and that the $O^*$-marking of $g$ corresponding to $v$ is in the rightmost square of the top row of $g'$ (Figure \ref{fig:sink}).

According to Proposition \ref{prop:tilde-hat}, it is sufficient to check that the homology of the tilde version vanishes.
Take a point $c=\alpha_{n+1}\cap\beta_{n+1}$.
We will decompose the set of states $\mathbf{S}(g)$ as the disjoint union $\mathbf{I}(g)\cup\mathbf{N}(g)$, where $\mathbf{I}(g)=\{\mathbf{x\in S}(g)|c\in\mathbf{x}\}$ and $\mathbf{N}(g)=\{\mathbf{x\in S}(g)|c\notin\mathbf{x}\}$.
This decomposition gives a decomposition $\widetilde{CF}(g,\omega)=\widetilde{I}\oplus \widetilde{N}$ as a vector space, where $\widetilde{I}$ and $\widetilde{N}$ are the spans of $\mathbf{I}(g),\mathbf{N}(g)$ respectively.
Then we can write the differential on $\widetilde{CF}(g,\omega)$ as
\[
\widetilde{\partial}=
\begin{pmatrix}
\partial_I^I & \partial_N^I\\
0 & \partial_N^N
\end{pmatrix},
\]
and $\widetilde{CF}(g,\omega)=\mathrm{Cone}(\partial_N^I\colon(\widetilde{N},\partial_N^N)\to(\widetilde{I},\partial_I^I))$.

To see that $\widetilde{HF}(g,\omega)=0$, we will check that the chain map $\partial_N^I\colon(\widetilde{N},\partial_N^N)\to(\widetilde{I},\partial_I^I)$ is a quasi-isomorphism.
Let $\mathbb{O}=\{O_0,O_1,\dots,O_n\}$ denote the set of $O$-markings of $g$ and $\mathbb{O}'=\{O_1,\dots,O_n\}$ denote the set of $O$-markings of $g'$.
We will think that $O_0$ is the $O^*$-marking in the topmost column of $g$ representing the sink.

Let $\mathcal{H}\colon\widetilde{I}\to\widetilde{N}$ and $\mathcal{H}_N\colon\widetilde{N}\to\widetilde{N}$ be two linear maps defined by for $\mathbf{x\in I}(g)$ and $\mathbf{y\in N}(g)$,
\begin{align*}
\mathcal{H}(\mathbf{x})=\sum_{\mathbf{y}'\in\mathbf{N}(g)}\#\{r\in\mathrm{Rect}^\circ(\mathbf{x,y'})|r\cap(\mathbb{O}\setminus O_0)=r\cap\mathbb{X}=\emptyset,O_0\in r\}\cdot\mathbf{y}',\\
\mathcal{H}_N(\mathbf{y})=\sum_{\mathbf{y}'\in\mathbf{N}(g)}\#\{r\in\mathrm{Rect}^\circ(\mathbf{y,y'})|r\cap(\mathbb{O}\setminus O_0)=r\cap\mathbb{X}=\emptyset,O_0\in r\}\cdot\mathbf{y}'.
\end{align*}
Then it is straightforward to see that $\partial_N^I\circ\mathcal{H}=\mathrm{id}_{\widetilde{I}}$ and $\mathcal{H}\circ\partial_N^I+\partial_N^N\circ\mathcal{H}_N+\mathcal{H}_N\circ\partial_N^N=\mathrm{id}_{\widetilde{N}}$ by counting all rectangles appearing in these equations.
So $\partial_N^I\colon(\widetilde{N},\partial_N^N)\to(\widetilde{I},\partial_I^I)$ is a quasi-isomorphism.
\end{proof}

\section{The preparations of Theorem \ref{thm:cutedge} (2)}
\label{sec:specialcpx}
\begin{figure}
\centering
\includegraphics[scale=0.5]{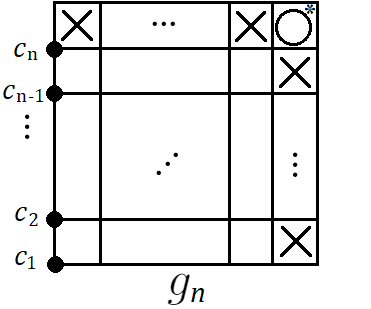}
\caption{The special grid-like diagram $C_n$}
\label{fig:specialC1}
\end{figure}
We will prepare a sequence of acyclic chain complexes.
These chain complexes will appear as subcomplexes of the grid chain complex.
We will introduce certain chain complexes using \textit{grid diagram-like diagrams}.
Let $n\geqq2$ be a fixed integer and $g_n$ be an $n\times n$ diagram as in the right of Figure \ref{fig:specialC1}.
Suppose that $g_n$ has exactly one $O^*$-marking and $2n-2$ $X$-markings and that there is no marking in the lower left $(n-1)\times(n-1)$ block.
Although $g_n$ is not a graph grid diagram, we can define a chain complex in the same manner as a graph grid diagram because $g_n$ has at least one $X$-marking in each row and column.

We define $\mathbb{O}^*,\mathbb{X},\mathbf{S}(g_n),\mathrm{Rect}^\circ$, and Maslov grading in the same manner as general grid homology in Section \ref{sec:dfn-of-grid-homology}.
We remark that $\#\mathbb{O}^*=1$ and $\#\mathbb{X}=2n-2$.
\begin{dfn}
\label{dfn:Cn}
The chain complex $C_n$ is an $\mathbb{F}$-vector space finitely generated by $\mathbf{S}(g_n)$, and whose differential is defined by
\[
\partial_n(\mathbf{x})=\sum_{\mathbf{y}\in\mathbf{S}(g)}\#\{r\in\mathrm{Rect}^\circ(\mathbf{x,y})|r\cap\mathbb{O}^*=r\cap\mathbb{X}=\emptyset\}\cdot\mathbf{y},
\]
where $\#\{\cdot\}$ counts rectangles modulo $2$.
\end{dfn}
By the definition of Maslov grading, $\partial_n$ drops a Maslov grading by one.

\begin{prop}
Let $C_n$ be a chain complex of Definition \ref{dfn:Cn}.
Then we have $H(C_n)=0$ for any $n\geqq2$. 
\end{prop}
\begin{proof}
We will prove by induction on $n$.

Let $n=2$, the chain complex $C_2$ is generated by two states $\mathbf{x}_{12}$ and $\mathbf{x}_{21}$, and we have
\[
\partial(\mathbf{x}_{12})=\mathbf{x}_{21},\ \partial(\mathbf{x}_{21})=0.
\]
So $H(C_2)=0$.

For the inductive step, we will divide $C_n$ into many subcomplexes isomorphic to $C_{n-1}$.
Take $n$ points $c_1,\dots,c_n$ on the vertical circle $\beta_1$ as in Figure \ref{fig:specialC1}.
For $i=1,\dots,n$, let $C_n(i)$ be the span of the set of states containing $c_i$.
Then we have $C_n=C_n(1)\oplus\dots\oplus C_n(n)$ as a vector space.

Let $r\in\mathrm{Rect}^\circ(\mathbf{x,y})$ be an empty rectangle and $c_j=\mathbf{x}\cap\beta_1$ and $c_k=\mathbf{y}\cap\beta_1$.
Then we have $j \geqq k$ because the rightmost column is filled by the markings.
\[
C_n(n) \subset C_n(n) \oplus C_n(n-1) \subset \dots \subset C_n(n)\oplus\dots\oplus C_n(1)=C_n.
\]
This sequence gives a sequence of chain complexes $C_n(n), C_n(n-1), \dots, C_n(1)$, where $C_n(i)$ is the quotient complex of $C_n(i) \oplus\dots\oplus C_n(n)$ by $C_n(i+1) \oplus\dots\oplus C_n(n)$ for $i=n-1,\dots,1$.

Obviously, we have $C_n(i)\cong C_{n-1}$ and $H(C_n(i))=0$ for $i=1,\dots,n$.
Applying Lemma \ref{lem:C/C'} verifies that $H(C_n)=0$.
\end{proof}

\section{The proof of theorem \ref{thm:cutedge} (2)}
Let $(f,\omega)$ be an MOY graph such that $f$ has a cut edge as a spatial graph (Definition \ref{dfn:cutedge}).
If at least one of two vertices connected by the cut edge is sink or source, then Theorem \ref{thm:cutedge} (1) shows that $\widehat{HF}(f,\omega)=0$.
Then suppose that the two vertices connected by the cut edge are neither sinks nor sources.

\begin{lem} 
\label{lem:C/C'}
Let $C$ be a chain complex and $C'$ be its subcomplex.
\begin{itemize}
    \item If $H(C')=0$, then $H(C)\cong H(C/C')$.
    \item If $H(C/C')=0$, then $H(C)\cong H(C')$.
\end{itemize}
\end{lem}
\begin{proof}
It follows immediately from the long exact sequence from the following short exact sequence:
\[
\xymatrix{
0\ar[r] & C' \ar[r] & C \ar[r] & C/C' \ar[r]  & 0.
}
\]
\end{proof}

For simplicity, we often use the following graph grid diagrams.
\begin{dfn}
\label{dfn:good-graph-diagram}
An $n \times n$ graph grid diagram is \textbf{good} if it satisfies the following two conditions:
\begin{itemize}
    \item The leftmost square of the top row and the rightmost square of the bottom row has an $O$- or $O^*$-marking,
    \item The rightmost square of the top row has an $X$-marking,
\end{itemize}
\end{dfn}

\subsection{The structure of the grid chain complex}
\label{sec:str-of-cpx}
Let $(f,\omega)$ be an MOY graph such that $f$ has a cut edge as a spatial graph and that the two vertices connected by it are neither sinks nor sources.
Then there exists a $2n\times 2n$ graph grid diagram $g$ for $f$.
Consider four $n \times n$ blocks obtained by cutting $g$ along the horizontal circles $\alpha_1\cup\alpha_{n+1}$ and the vertical circles $\beta_1\cup\beta_{n+1}$.
We will call the four blocks $g_{11}$, $g_{12}$, $g_{21}$, and $g_{22}$ respectively (see Figure \ref{fig:g_withcutedge}).
We can assume that $g$ satisfies the following conditions:
\begin{itemize}
\item $g_{11}$ and $g_{22}$ can be viewed as good graph grid diagrams.
\item The rightmost square of the bottom row of $g_{11}$ has an $O^*$-marking.
\item $g_{12}$ has no $O$- or $O^*$-markings and only one $X$-marking in the leftmost square of the bottom row.
\item $g_{21}$ has no markings.
\item The leftmost square of the top row of $g_{22}$ has an $O^*$-marking.
\end{itemize}

\begin{figure}
\centering
\includegraphics[scale=0.4]{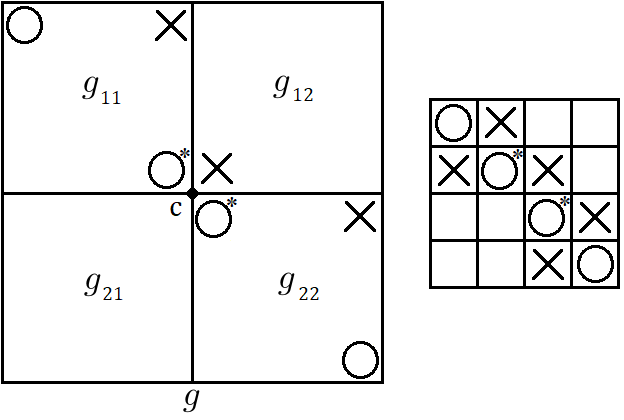}
\caption{Left: A $2n\times2n$ graph grid diagram. Right: an example of $2n\times2n$ diagram for a spatial handcuff graph.}
\label{fig:g_withcutedge}
\end{figure}

Take a point $c$ as the intersection point $\alpha_{n+1}\cap\beta_{n+1}$.
Using the same notations as the proof of Theorem \ref{thm:cutedge} (1), we can write $\widetilde{CF}(g,\omega)$ as a mapping cone $\mathrm{Cone}(\partial_I^N\colon (\widetilde{I},\partial_I)\to(\widetilde{N},\partial_N))$.

For $k=1,\dots,n$, let $\mathbf{S}(g_{11},k)$ be a $k$-tuple of points of $(\alpha_{n+1}\cup\dots\cup\alpha_{2n})\cap(\beta_1\cup\dots\cup\beta_n)$ so that each horizontal and vertical circle has at most one point.
We define $\mathbf{S}(g_{12},k)$, $\mathbf{S}(g_{21},k)$, and $\mathbf{S}(g_{22},k)$ in the same way.
Then we will represent each state $\mathbf{x\in S}(g)$ uniquely as
\[
\mathbf{x}=
\begin{pmatrix}
\mathbf{x}_{11} & \mathbf{x}_{12} \\
\mathbf{x}_{21} & \mathbf{x}_{22} \\
\end{pmatrix}
,
\]
where $\mathbf{x}_{11}\in\mathbf{S}(g_{11},n-k)$, $\mathbf{x}_{12}\in\mathbf{S}(g_{12},k)$, $\mathbf{x}_{21}\in\mathbf{S}(g_{21},k)$, and $\mathbf{x}_{22}\in\mathbf{S}(g_{22},n-k)$.
Using this representation, we decompose the set of grid states $\mathbf{S}(g)$ as disjoint union $\mathbf{S}_0(g)\cup\dots\cup\mathbf{S}_n(g)$, where $\mathbf{S}_k(g)$ is the set of states represented by $\mathbf{S}(g_{11},n-k)$, $\mathbf{S}(g_{12},k)$, $\in\mathbf{S}(g_{21},k)$, and $\mathbf{S}(g_{22},n-k)$.

Now we give the splitting of the vector space 
\[
\widetilde{CF}(g,\omega) \cong \left(\bigoplus_{k=1}^n I_k \right) \oplus \left(\bigoplus_{k=0}^n N_k \right),
\]
where $I_k$ is the span of $\mathbf{S}_k(g)\cap\mathbf{I}(g)$, and $N_k$ is the span of $\mathbf{S}_k(g)\cap\mathbf{N}(g)$.
We remark that $\mathbf{S}_0(g)\cap\mathbf{I}(g)=\emptyset$.

The differential of $\widetilde{CF}(g,\omega)$, denoted by $\widetilde{\partial}$, satisfies that for $k=1,\dots,n$,
\begin{align*}
\widetilde{\partial}(I_k) &\subset I_{k-1}\oplus I_{k} \oplus N_{k-1},\\
\widetilde{\partial}(N_k) &\subset N_{k-1}\oplus N_{k}.
\end{align*}
This relation can be expressed by the following schematic picture:
\[
\xymatrix{
I_1 \ar[d] & I_2 \ar[d] \ar[l] & \cdots\ar[l] & I_n \ar[d] \ar[l] & {}\\
N_0  & N_1 \ar[l] & \cdots\ar[l] &  N_{n-1}\ar[l]  & N_{n}\ar[l]
}
\]
% \[
% \xymatrix{
% I_1 \ar[d]_-{\partial_I^N} & I_2 \ar[d]_-{\partial_I^N} \ar[l]_-{\partial_I} & \cdots\ar[l]_-{\partial_I} & I_n \ar[d]_-{\partial_I^N} \ar[l]_-{\partial_I} & {}\\
% N_0  & N_1 \ar[l]_-{\partial_N} & \cdots\ar[l]_-{\partial_N} &  N_{n-1}\ar[l]_-{\partial_N}  & N_{n}\ar[l]_-{\partial_N}
% }
% \]
The top row of the picture represents the chain complex $(\widetilde{I},\partial_I)$, and the bottom row represents the chain complex $(\widetilde{N},\partial_N)$.
%This picture implies that we can consider $I_1, \dots, I_n$ and $N_0,\dots, N_n$ as chain complexes and the arrows as chain maps.

\subsection{The main idea of the proof}
\label{sec:proof-of-main-thm}
The main idea of the proof is to see that the following chain complexes are acyclic:
\begin{itemize}
\item $\mathcal{C}=\mathrm{Cone}(\partial_I^N|_{I_1}\colon I_1\to N_0)$, which is the subcomplex of $\widetilde{CF}(g,\omega)$,
\item $N_k$, which is the subcomplex of $\widetilde{CF}(g,\omega)/(\mathcal{C}\oplus N_1\oplus\cdots\oplus N_{k-1})$ for $k=1,\dots,n$,
\item $I_k$, which is the subcomplex of $\widetilde{CF}(g,\omega)/(\widetilde{N}\oplus I_{2}\oplus\dots\oplus I_{k-1})$ for $k=2,\dots,n-1$,
\item $I_n$, which is the quotient complex $\widetilde{CF}(g,\omega)/(\widetilde{N}\oplus I_{2}\oplus\dots\oplus I_{n-1})$.
\end{itemize}
Then Lemma \ref{lem:C/C'} verifies Theorem \ref{thm:cutedge} (2).
To see these, we will observe that these chain complexes are decomposed into finitely many acyclic chain complexes defined in Section \ref{sec:specialcpx}.

\begin{lem}
\label{lem:C-is-acyclic}
$\mathcal{C}$ is acyclic.
\end{lem}
\begin{proof}
Since any state $\mathbf{x}\in\widetilde{I}$ satisfies $\widetilde{\partial}\circ \widetilde{\partial}(\mathbf{x})=(\partial_I \circ \partial_I +\partial_I^N \circ \partial_I + \partial_N \circ \partial_I^N)(\mathbf{x})=0$, we have $\partial_I^N\circ \partial_I+\partial_N\circ\partial_I^N=0$ and thus $\partial_I^N$ is a chain map.

For any state $\mathbf{x}$ of $\widetilde{I}$, the chain map $\partial_I^N|_{I_1}$ only counts the rectangle that has $c$ and $\mathbf{x}_{21}$ as its corners.
So $\partial_I^N|_{I_1}$ is a bijection and thus an isomorphism.

Since $\partial_I^N|_{I_1}$ is an isomorphism of chain complexes, the usual argument of mapping cone deduces that $H(\mathcal{C})=0$.
\end{proof}

To observe $I_k$ and $N_k$, we will introduce the following.
\begin{dfn}
For a state $\mathbf{x}=\begin{pmatrix}
\mathbf{x}_{11} & \mathbf{x}_{12} \\
\mathbf{x}_{21} & \mathbf{x}_{22} \\
\end{pmatrix}$ of $N_k$ or $I_k$ $(k=2,\dots,n)$, the \textbf{modified Maslov grading} $M'$ for $\mathbf{x}$ is defined by $M'(\mathbf{x})=M(\mathbf{x}_{11})+M(\mathbf{x}_{22})$.
\end{dfn}
%We remark that for not only a state but also a set of the finite points on the planar realization of $g$, $M$ can be defined by (\ref{eq:Maslov}).

\begin{lem}
\label{lem:M'}
The differentials of $N_2,\dots, N_n$ and $I_2,\dots, I_n$ preserve or drop the modified Maslov grading.
Moreover, the modified Maslov grading is preserved if and only if the differential does not change $\mathbf{x}_{11}$ or $\mathbf{x}_{22}$.
\end{lem}

\begin{proof}

Let $r$ be an empty rectangle counted by the differential.
Then $r$ does not change $\mathbf{x}_{11}$ and $\mathbf{x}_{22}$ simultaneously.
If both $\mathbf{x}_{11}$ and $\mathbf{x}_{22}$ are preserved, then the modified Maslov grading is preserved.

Suppose that $r$ changes $\mathbf{x}_{11}$ and preserve $\mathbf{x}_{22}$.
Then there are three cases of $r$:

\begin{enumerate}[(i)]
\item $r$ changes only $\mathbf{x}_{11}$.
It is clear that the modified Maslov grading drops.

\item $r$ changes only $\mathbf{x}_{11}$ and $\mathbf{x}_{21}$.
Suppose that $r \cap \alpha_{n+1} \neq \emptyset$.
Let $x_1$ and $x_2$ be the intersection points at the northeast and northwest corners of $r$ respectively.
Let $\beta_{i}$ be the vertical circle containing $x_2$, and $\beta_{i+k}$ be the vertical circle containing $x_1$.
Then $\mathcal{J}(\mathbf{x}_{11},\mathbb{O})$ drops by $k$ because there are $k$ $O$-markings above the segment connecting $x_1$ and $x_2$.
On the other hand $\mathcal{J}(\mathbf{x}_{11},\mathbf{x}_{11})$ drops at most $k-1$ because there are at most $k-1$ points above the segment.
Therefore $M'$ drops.
The cases that $r \cap \alpha_{1} \neq \emptyset$ can be shown similarly.

\item $r$ changes only $\mathbf{x}_{11}$ and $\mathbf{x}_{12}$.
The same argument shows that $M'$ drops.
\end{enumerate}

The case that $r$ changes $\mathbf{x}_{22}$ and preserve $\mathbf{x}_{11}$ is the same.
\end{proof}

\begin{lem}
\label{lem:I2-Nn-is-acyclic}
$N_k$ and $I_k$ are acyclic for $k=2,\dots, n$.
\end{lem}

\begin{proof}
We will show that $N_2, \dots, N_n$ are acyclic.
$I_2,\dots,I_n$ can be shown in the same way.
Let $m=\mathrm{min}\{ M'(\mathbf{x}) | \mathbf{x} \in \mathbf{S}_k(g) \cap \mathbf{N}(g) \}$ and $M=\mathrm{max}\{ M'(\mathbf{x}) | \mathbf{x} \in \mathbf{S}_k(g) \cap \mathbf{N}(g) \}$.
We obtain the splitting of the vector space $N_k=\bigoplus_{i=m}^M N_{k}^i$, where $N_k^i$ is the span of the grid states whose the modified Maslov grading is $i$.

According to Lemma \ref{lem:M'}, we have a sequence of subcomplexes,
\[
N_k^m \subset (N_k^{m} \oplus N_k^{m+1}) \subset (N_k^{m} \oplus N_k^{m+1} \oplus N_k^{m+2}) \subset \dots \subset \bigoplus_{i=m}^M N_{k}^i=N_k.
\]
This sequence deduces a sequence of chain complexes ${N}_k^m,\dots,N_k^M$, where $N_k^i$ is the quotient complex of $(N_k^{m} \oplus\dots\oplus N_k^i)$ by $(N_k^{m} \oplus\dots\oplus N_k^{i-1})$ for $i=m+1, \dots, M$.

Lemma \ref{lem:C/C'} implies that it is sufficient to see that the homology of $N_k^i$ vanishes for $i=m, \dots, M$.

For $i=m,\dots M$, let $\mathbf{S}_k(g|M'=i)$ be the set of pairs $\{(\mathbf{x}_{11}, \mathbf{x}_{22})\in \mathbf{S}(g_{11},n-k) \times \mathbf{S}(g_{22},n-k)|M(\mathbf{x}_{11})+M(\mathbf{x}_{22})=i\}$.

We have the decomposition of the vector space
\begin{align}
\label{eq:N_k^i}
N_k^i=\bigoplus_{(\mathbf{y}_{11}, \mathbf{y}_{22})\in \mathbf{S}_k(g|M'=i)} N(\mathbf{y}_{11}, \mathbf{y}_{22}),
\end{align}

where $N(\mathbf{y}_{11}, \mathbf{y}_{22})$ is the span of the set of states 
$\begin{pmatrix}
\mathbf{x}_{11} & \mathbf{x}_{12} \\
\mathbf{x}_{21} & \mathbf{x}_{22} \\
\end{pmatrix}\in \mathbf{S}_k(g)\cap\mathbf{N}(g)$ with $\mathbf{x}_{11}=\mathbf{y}_{11}$ and $\mathbf{x}_{22}=\mathbf{y}_{22}$.
Again using Lemma \ref{lem:M'}, we can regard (\ref{eq:N_k^i}) as a decomposition of the chain complex.
Clearly each summand $N(\mathbf{y}_{11}, \mathbf{y}_{22})$ is isomorphic to $C_k\otimes C(\mathbf{y}_{11}, \mathbf{y}_{22})$, where $C_k$ is the special chain complex in Section \ref{sec:specialcpx} and $C(\mathbf{y}_{11}, \mathbf{y}_{22})$ is some chain complex.
The points of $g_{21}$ are corresponding to $C_k$ and $g_{12}$ to $C(\mathbf{y}_{11}, \mathbf{y}_{22})$.
Therefore we have $H(N(\mathbf{y}_{11}, \mathbf{y}_{22}))=0$ and hence $H(N_k^i)=0$.
Lemma \ref{lem:C/C'} shows that $H(N_k)=0$.

The same argument shows that $I_k$ is decomposed into many copies of $C_{k-1}\otimes C_k$, and $H(I_k)=0$ follows.
We remark that $I_2$ is decomposed into many copies of $C_2$.
\end{proof}

\begin{lem}
\label{lem:N1-is-acyclic}
$N_1$ is acyclic.
\end{lem}

\begin{proof}
The main idea is the same as the previous lemma.
Let $\mathbf{x}=
\begin{pmatrix}
\mathbf{x}_{11} & \mathbf{x}_{12} \\
\mathbf{x}_{21} & \mathbf{x}_{22} \\
\end{pmatrix}$ be a state of $\mathbf{S}_1(g)\cap\mathbf{I}(g)$.
\begin{figure}
\centering
\includegraphics[scale=0.5]{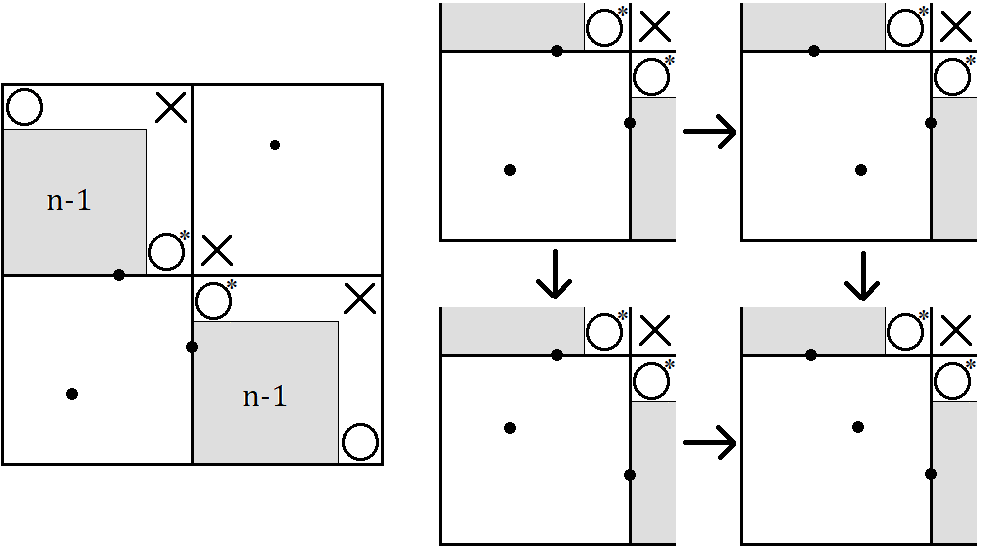}
\caption{Left: A state of case (1). Right: an acyclic complex consisting of four states. The arrows represent the differential.}
\label{fig:N_1-case1}
\end{figure}
\begin{figure}
\centering
\includegraphics[scale=0.5]{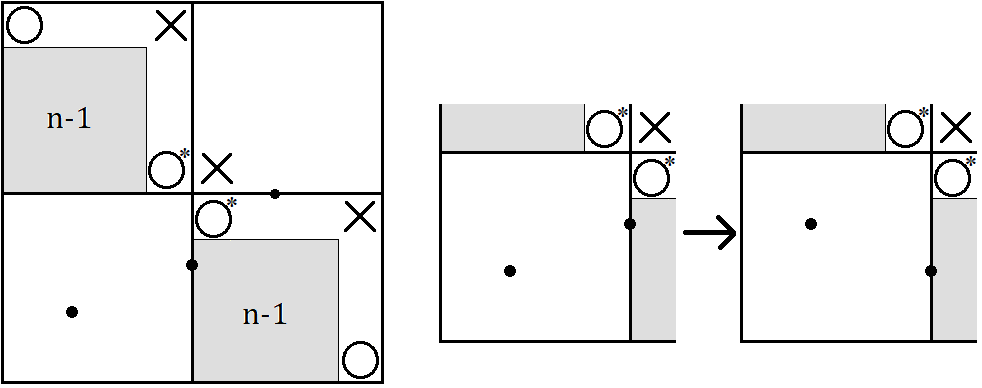}
\caption{Left: A state of case (2). Right: an acyclic complex consisting of two states. The arrow represents the differential.}
\label{fig:N_1-case2}
\end{figure}
Then we have the following three cases;
\begin{enumerate}[(1)]
    \item The point $\mathbf{x}_{12}$ is not on $\alpha_{n+1}$ or $\beta_{n+1}$.
    The collection of states satisfying this, say $N_1'$, forms a subcomplex of $N_1$. 
    Then let $x'=\mathbf{x}_{11}\cap\alpha_{n+1}$ and $x''=\mathbf{x}_{22}\cap\beta_{n+1}$ and define $M'(\mathbf{x})=M((\mathbf{x}_{11}\setminus\{x'\})\cup\mathbf{x}_{12}\cup(\mathbf{x}_{22}\setminus\{x''\}))$.
    Then $M'$ satisfies the same property as Lemma \ref{lem:M'}, in other words, $M'$ is preserved if and only if the differential does not change $\mathbf{x}_{11}\setminus\{x'\}$, $\mathbf{x}_{12}$, or $\mathbf{x}_{22}\setminus\{x''\}$.
    Then consider the four states such that all of $\mathbf{x}_{11}\setminus\{x'\}$, $\mathbf{x}_{12}$, $\mathbf{x}_{22}\setminus\{x''\}$ coincide and that the values of $M'$ are the minimum.
    These four states form the subcomplex of $N_1$ (Figure \ref{fig:N_1-case1}).
    Figure 4 directly shows this type of complex is acyclic.
    The same argument as Lemma \ref{lem:I2-Nn-is-acyclic} shows that $N_1'$ is acyclic. 
    \item The point $\mathbf{x}_{12}$ is on the horizontal line $\alpha_{n+1}$.
    Then let $x'=\mathbf{x}_{22}\cap\beta_{n+1}$ and define $M'(\mathbf{x})=M(\mathbf{x}_{11}\cup\mathbf{x}_{12}\cup\mathbf{x}_{22}\setminus\{x'\})$.
    Then the same argument as the previous case shows that the collection of the states of this case is acyclic.
    In this case, we have acyclic complexes consisting of two states (Figure \ref{fig:N_1-case2}).
    \item The point $\mathbf{x}_{12}$ is on the vertical line $\beta_{n+1}$.
    Then let $x'=\mathbf{x}_{11}\cap\alpha_{n+1}$ and define $M'(\mathbf{x})=M((\mathbf{x}_{11}\setminus\{x'\})\cup\mathbf{x}_{12}\cup\mathbf{x}_{22})$.
    Then the same argument as the case (2) concludes the proof.
\end{enumerate}
\end{proof}

\section{The proof of Corollary \ref{cor:wedge sum}}
For two MOY graphs $(f_1,\omega_1)$ and $(f_2,\omega_2)$, let $(f\vee_{(v_1,v_2)} f_2,\omega_1\vee\omega_2)$ be the spatial graph (Definition \ref{dfn: disjoint-connected-wedge-sum}) and $f$ be the transverse spatial graph consisting of $f_1 \sqcup f_2$ and a cut edge from $v_1$ to $v_2$.
Take a balanced coloring $\omega$ for $f$ naturally determined by $\omega_1$ and $\omega_2$.

Let $g$ be a $2n\times 2n$ graph grid diagram for $f$ and $g_{\vee}$ be an $(2n-1)\times (2n-1)$ graph grid diagram for $f\vee_{(v_1,v_2)} f_2$ as in Figure \ref{fig:g-gvee}.
\begin{figure}
\centering
\includegraphics[scale=0.4]{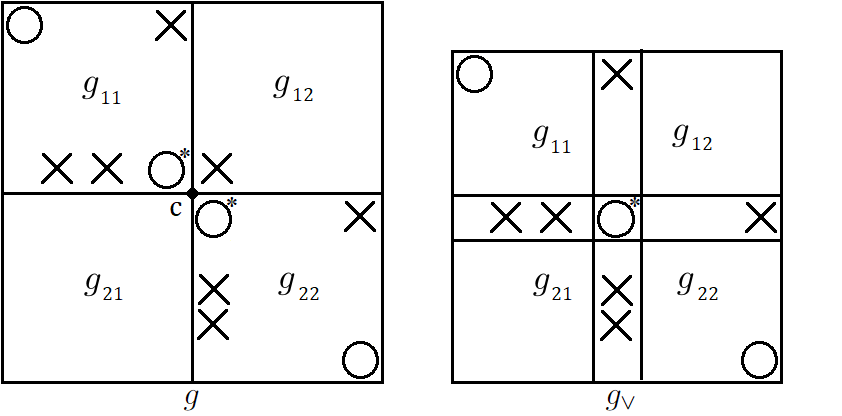}
\caption{Two graph grid diagrams $g$ and $g_{\vee}$}
\label{fig:g-gvee}
\end{figure}
Let $g_{ij}$ $(i,j\in\{1,2\})$ be four $n\times n$ blocks of $g$ as in Figure \ref{fig:g-gvee}. 
We can assume the following conditions:
\begin{itemize}
    \item $g_{11}$ and $g_{22}$ can be viewed as good graph grid diagrams,
    \item $g_{11}$ and the upper left $n \times n$ block of $g_{\vee}$ are the same,
    \item $g_{22}$ and the lower right $n \times n$ block of $g_{\vee}$ are the same,
    \item $g_{12}$ has no $O$- or $O^*$-marking and exactly one $X$-marking in its the leftmost square of the bottom row,
    \item $g_{21}$ has no markings,
    \item The upper right and the lower left $(n-1) \times (n-1)$ blocks of $g_{\vee}$ have no markings.
\end{itemize}
Since $g$ represents a transverse spatial graph with a cut edge, Section \ref{sec:str-of-cpx} implies that the chain complex $\widetilde{CF}(g,\omega)$ can be written as $\mathrm{Cone}(\partial_I^N\colon \widetilde{I}\to\widetilde{N})$.

The following Lemma will be used sometimes:
\begin{lem}
\label{lem:N0-CFxCF}
Let $g$ be a $2n \times 2n$ graph grid diagram and $g_1$, $g_2$ be two $n \times n$ graph grid diagram.
Suppose that they satisfy the following conditions:
\begin{itemize}
    \item $g_1$ and $g_2$ are good graph diagrams,
    \item $g_1$ and the upper left $n \times n$ block of $g$ are the same,
    \item $g_2$ and the lower right $n \times n$ block of $g$ are the same,
    \item The lower left $n \times n$ block has no markings.
\end{itemize}
Let $\omega$ be a weight for $g$ and $\omega_i$ be a weight for $g_i$ given by the restriction of $\omega$.
Then for $N_0$, which is a subcomplex of $\widetilde{CF}(g,\omega)$ (Section \ref{sec:str-of-cpx}), there is an isomorphism
\[
N_0 \cong \widetilde{CF}(g_1,\omega_1) \otimes \widetilde{CF}(g_2,\omega_2)\llbracket-1,0\rrbracket.
\]
\end{lem}
\begin{proof}
Let $\mathbf{x}=\begin{pmatrix}
\mathbf{x}_{11} & \mathbf{x}_{12} \\
\mathbf{x}_{21} & \mathbf{x}_{22} \\
\end{pmatrix}$ be a state of $N_0$.
Since the upper left and lower right $n\times n$ blocks are good graph diagrams, every empty rectangle from $\mathbf{x}$ does not change $\mathbf{x}_{11}$ and $\mathbf{x}_{22}$ simultaneously.
So the natural correspondence $\mathbf{x}\to \mathbf{x}_{11} \otimes \mathbf{x}_{22}$ induces an isomorphism $N_0\cong \widetilde{CF}(g_1,\omega_1)\otimes \widetilde{CF}(g_2,\omega_2)\llbracket-1,0\rrbracket$.
The direct computation shows that this isomorphism increases the Maslov grading by one and preserves the Alexander grading.
\end{proof}

\begin{proof}[proof of Corollary \ref{cor:wedge sum}]
Since $g$ represents a transverse spatial graph with a cut edge, Lemmas \ref{lem:C-is-acyclic}-\ref{lem:I2-Nn-is-acyclic} and \ref{lem:N0-CFxCF} implies that $H(\widetilde{I})\cong H(I_1)\cong H(N_0)\llbracket1,0\rrbracket\cong \widetilde{HF}(g_1,\omega_1) \otimes \widetilde{HF}(g_2,\omega_2)$.

For a state $\mathbf{x}\cup\{c\}$ of $\widetilde{I}$, let $\phi\colon \widetilde{I}\to \widetilde{CF}(g_{\vee}, \omega_1\vee\omega_2)$ be a linear map defined by $\phi(\mathbf{x}\cup\{c\})=\mathbf{x}$.
The same argument as \cite[Lemma 5.2.5]{grid-book} shows that $\phi$ is an isomorphism of absolute Maslov graded, relative Alexander graded chain complexes.
Therefore we have 
\[
\widetilde{HF}(g_{\vee}, \omega_1\vee\omega_2)\cong \widetilde{HF}(g_1,\omega_1) \otimes \widetilde{HF}(g_2,\omega_2).
\]
Proposition \ref{prop:tilde-hat} gives 
\[
\widehat{HF}(g_{\vee}, \omega_1\vee\omega_2)\cong \widehat{HF}(g_1,\omega_1) \otimes \widehat{HF}(g_2,\omega_2).
\]
\end{proof}

\section{The proof of Theorem \ref{cor:connected-sum-knot}}
We will show the case that $L_1$ and $L_2$ are knots.
The same argument holds for the case of two links.

Let $(f_1,\omega_1)$ and $(f_2,\omega_2)$ be two MOY graph.
Suppose that there is a pair $(v_1,v_2)\in V(f_1)\times V(f_2)$ such that $\omega_1(v_1)=\omega_2(v_2)$.
Let $f$ be a transverse spatial graph consisting of $f_1 \sqcup f_2$ and a cut edge from $v_1$ to $v_2$.
Take a balanced coloring $\omega$ for $f$ naturally determined by $\omega_1$ and $\omega_2$.
\begin{figure}
\centering
\includegraphics[scale=0.4]{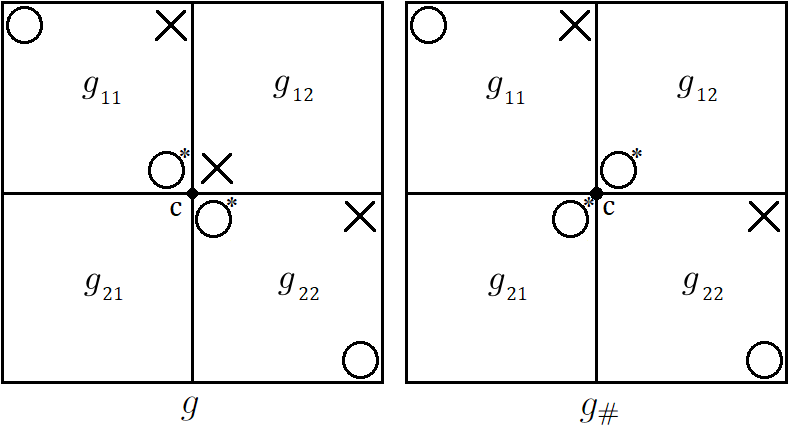}
\caption{Graph grid diagrams for $f$ and $f_1\#_{(v_1,v_2)} f_2$}
\label{fig:cut-connected}
\end{figure}
Then we can take two $2n \times 2n$ graph grid diagrams $g$ and $g_{\#}$ for $f$ and $f_1\#_{(v_1,v_2)}f_2$ respectively as in Figure \ref{fig:cut-connected}.
Let $c$ be the intersection point $\alpha_{n+1}\cap\beta_{n+1}$ on $g$ and $g_{\#}$.
We can assume that $g$ and $g_{\#}$ coincide except for the $2\times 2$ block around $c$.

We decompose the set of states as $\mathbf{S}(g_{\#})=\mathbf{I}(g_{\#})\sqcup\mathbf{N}(g_{\#})$, where $\mathbf{I}(g_{\#})$ is the set of states containing $c$.
Using the spans of them, we obtain the splitting of the vector space $\widetilde{CF}(g_{\#})\cong \widetilde{N}_{\#} \oplus \widetilde{I}_{\#}$.
Then we can write the chain complex of $g_{\#}$ as $\widetilde{CF}(g_{\#})=\mathrm{Cone}(\partial_N^I \colon \widetilde{N}_{\#} \to \widetilde{I}_{\#})$, where $\partial_N^I$ is the chain map counting empty rectangles from a state of $\mathbf{N}(g_{\#})$ to a state of $\mathbf{I}(g_{\#})$.

Since $g$ represents a transverse spatial graph with a cut edge, Section \ref{sec:str-of-cpx} implies that its chain complex is written as $\widetilde{CF}(g)=\mathrm{Cone}(\partial_I^N \colon \widetilde{I} \to \widetilde{N})$.
Recall that $\widetilde{I}$ has a subcomplex $I_0$ and $\widetilde{N}$ has a subcomplex $N_1$ such that $H(\widetilde{I})\cong H(I_0)$ and $H(\widetilde{N})\cong H(N_0)$.

\begin{lem}
\label{lem:Ihash-Nhash}
As absolute Maslov graded, relative Alexander graded chain complexes, there are natural isomorphisms $\widetilde{I}\cong \widetilde{I}_{\#}\llbracket 1,\omega_1(v_1)\rrbracket$ and $\widetilde{N}\cong \widetilde{N}_{\#}\llbracket -1,0\rrbracket$.
\end{lem}
\begin{proof}
Natural bijections $\mathbf{I}(g)\to\mathbf{I}(g_{\#})$ and $\mathbf{N}(g)\to\mathbf{N}(g_{\#})$ give isomorphisms because for any empty rectangle $r$ counted by the differential of $\widetilde{I}$, $\widetilde{N}$, $\widetilde{I}_{\#}$, and $\widetilde{N}_{\#}$, $r$ is disjoint from the interior of $2\times 2$ block around $c$.

A simple computation shows that the bijection $\mathbf{I}(g)\to\mathbf{I}(g_{\#})$ drops the Maslov grading by one and the Alexander grading by $\omega(v_1)$, and the bijection $\mathbf{N}(g)\to\mathbf{N}(g_{\#})$ increase the Maslov grading by one and preserve the Alexander gradings.
\end{proof}

\begin{lem}
\label{lem:partial-hash-trivial}
The induced map on homology $H(\partial_{\#})\colon H(\widetilde{N}_{\#}) \to H(\widetilde{I}_{\#})$ is trivial.
\end{lem}
\begin{proof}
Using the isomorphism $\widetilde{N} \cong \widetilde{N}_{\#}\llbracket -1,0\rrbracket$, let $N_{\# 0}$ be a subcomplex of $\widetilde{N}_{\#}$ which is isomorphic to $N_0$.
Then we have $H(\widetilde{N}_{\#})\cong H(N_{\# 0})$ since $H(\widetilde{N})\cong H(N_{0})$.
To show this Lemma, it is sufficient to see that $\partial_{\#}(N_{\#0})=0$.

Let $r$ be a rectangle from a state of $N_0$ to a state of $\widetilde{I}$.
Then $r$ must contain one of $X$-marking drawn in the right of Figure \ref{fig:cut-connected}.
Therefore we have $\partial_{\#}(N_{\#0})=0$.
\end{proof}

\begin{proof}[proof of Theorem \ref{thm:connected-sum}]
Lemma \ref{lem:partial-hash-trivial} implies 
\[
\widetilde{HF}(f_1\#_{(v_1,v_2)}f_2,\omega_1\#\omega_2) \cong H(\widetilde{I}_{\#}) \oplus H(\widetilde{N}_{\#}).
\]

Using Lemmas \ref{lem:C-is-acyclic}-\ref{lem:I2-Nn-is-acyclic}, \ref{lem:N0-CFxCF}, and \ref{lem:Ihash-Nhash}, we have 
\[\widetilde{HF}(f_1\#_{(v_1,v_2)}f_2,\omega_1\#\omega_2) \cong \widetilde{HF}(g_1,\omega_1)\otimes \widetilde{HF}(g_2,\omega_2) \otimes W(\omega_1(v_1)).
\]
Then Proposition \ref{prop:tilde-hat} gives 
\[
\widehat{HF}(f_1\#_{(v_1,v_2)}f_2,\omega_1\#\omega_2) \cong \widehat{HF}(g_1,\omega_1)\otimes \widehat{HF}(g_2,\omega_2) \otimes W(\omega_1(v_1)).
\]
\end{proof}

\begin{proof}[proof of Corollary \ref{cor:connected-sum-knot}]
We will regard a knot as a transverse spatial graph consisting of one vertex and edge.
If we think of a balanced coloring for a knot that sends the only edge to one, our grid homology $\widehat{HF}$ coincides with the original grid homology $\widehat{GH}$, and thus with knot Floer homology $\widehat{HFK}$ up to shift of the Alexander grading.
Since the Alexander grading of $\widehat{GH}$ and $\widehat{HFK}$ only depends on the knot type, it is sufficient to prove the connected sum formula for $\widehat{HF}$.

Now we use only balanced colorings that send the edges to one, so we write $\widehat{HF}(f)$ instead of $\widehat{HF}(f,\omega)$.
For $i=1,2$, let $v_i$ be the only vertex of $K_i$.
Then $K_1\#_{(v_1,v_2)}K_2$ (Definition \ref{dfn: disjoint-connected-wedge-sum}) is a transverse spatial graph consisting of two vertices and two edges.
By contracting one of two edges, we can obtain a transverse spatial graph corresponding to $K_1\# K_2$.

Theorem \ref{thm:connected-sum} implies
\[
\widehat{HF}(K_1\#_{(v_1,v_2)}K_2) \cong \widehat{HF}(K_1) \otimes \widehat{HF}(K_2) \otimes W(1).
\]
as absolute Maslov graded, relative Alexander graded vector space.
Contracting one edge of $K_1\#_{(v_1,v_2)}K_2$ yields a transverse spatial graph corresponding to $K_1\# K_2$.
By \cite[Theorem 1.9]{grid-MOY}, we have
\[
\widehat{HF}(K_1\# K_2)\cong \widehat{HF}(K_1) \otimes \widehat{HF}(K_2),
\]
as absolute Maslov graded, relative Alexander graded vector space.
Then the connected sum formula for $\widehat{GH}$ and $\widehat{HFK}$ follows.
\end{proof}

\section{The proof of Theorem \ref{thm:disjoint}}

%\begin{lem}
%Let $f$ be a spatial graph.
%For any two balanced colorings $\omega$ and $\omega'$ for $f$, we have
%\[
%\mathrm{dim}_{\mathbb{F}}\widehat{HF}(f, \omega)=\mathrm{dim}_{\mathbb{F}}\widehat{HF}(f, \omega')
%\]
%\end{lem}
%\begin{proof}
%Take a graph grid diagram for $f$.
%Then $\widetilde{CF}(g,\omega)$ and $\widetilde{CF}(g,\omega')$ are naturally isomorphic when we forget their Alexander gradings.
%\end{proof}

Let $g_1$ and $g_2$ be two $n \times n$ graph grid diagrams for $(f_1,\omega_1)$ and $(f_1,\omega_2)$ respectively.
Then there is a natural $2n \times 2n$ graph grid diagram $g_{\sqcup}$ for $f_1\sqcup f_2$ using $g_1$ and $g_2$.
Take the $(2n+4)\times (2n+4)$ graph grid diagram $g'$ obtained from $g_{\sqcup}$ by adding two rows and columns as in Figure \ref{fig:g-sqcup}.
Let $\omega_{\sqcup}$ be a weight for $g_{\sqcup}$ naturally determined by $\omega_1$ and $\omega_2$.
Let $\omega'$ be a weight for $g'$ sending the marking representing the two unknots to one and the others to the same integers as $\omega_{\sqcup}$.

$g'$ represents the spatial graph consisting of the disjoint union of $f_1$, $f_2$, and two unknots.
As in the left of Figure \ref{fig:g_withcutedge}, let $g_{ij}$ $(i,j\in\{1,2\})$ be four $(n+2)\times(n+2)$ blocks obtained by cutting $g'$ along $\alpha_1$, $\alpha_{n+3}$, $\beta_1$, and $\beta_{n+3}$.

\begin{figure}
\includegraphics[scale=0.4]{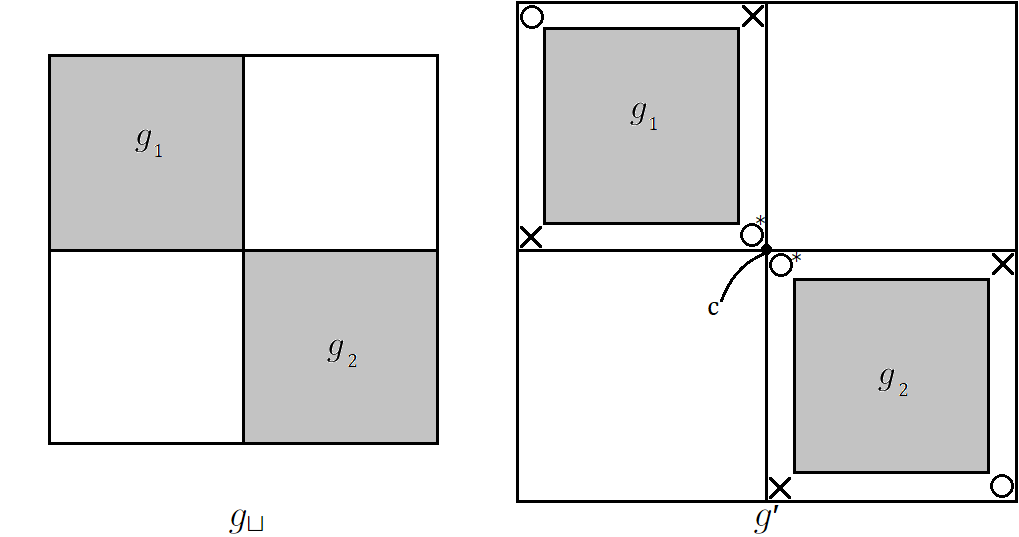}
\caption{The graph grid diagrams $g_{\sqcup}$ and $g'$.}
\label{fig:g-sqcup}
\end{figure}

The following Lemma is quickly proved as the extension of \cite[Lemma 8.4.2]{grid-book} using the same argument.
\begin{lem}
\label{lem:adding-unknot}
Let $(f,\omega)$ be an MOY graph.
Let $(\mathcal{O}, \omega_{\mathcal{O}})$ be the MOY graph where $\mathcal{O}$ is the unknot consisting of one vertex and edge, and $\omega_{\mathcal{O}}$ sends the edge of $\mathcal{O}$ to one.
Then there is an isomorphism of absolute Maslov relative Alexander graded $\mathbb{F}$-vector spaces
\begin{equation}
\label{eq:add1unknot}
\widehat{HF}(f\sqcup\mathcal{O},\omega\sqcup\omega_{\mathcal{O}})\cong \widehat{HF}(f,\omega)\otimes W(0).
\end{equation}

\end{lem}
\begin{proof}
In \cite[Lemmas 8.4.2 and 8.4.6]{grid-book}, they constructed quasi-isomorphisms and chain homotopy equivalences for the minus version.
Consider the induced maps on our hat version.
Then the analogies of the proof of \cite[Lemmas 8.4.2 and 8.4.6]{grid-book} prove (\ref{eq:add1unknot}).
\end{proof}

Let $c$ be the intersection point $\alpha_{n+3}\cap\beta_{n+3}$ on $g'$.
Using the same notations as Section \ref{sec:str-of-cpx}, we can write the grid chain complex of $g'$ as $\widetilde{CF}(g',\omega')=\mathrm{Cone}(\partial'\colon \widetilde{I} \to \widetilde{N})$, where $\partial'$ is the chain map counting empty rectangles from a state of $\widetilde{I}$ to a state of $\widetilde{N}$.

We will examine the structure of subcomplexes $N_0$ and $I_1$.
Take two points $d=\alpha_{n+3}\cap\beta_1$ and $e=\alpha_{1}\cap\beta_{n+3}$ on $g'$.
We decompose the set of states $\mathbf{S}(g')$ as the disjoint union $\mathbf{S}(g')=II\sqcup IN \sqcup NI \sqcup NN$, where
\begin{align*}
II &=\{\mathbf{x\in S}(g)|d,e\in\mathbf{x}\},\\
IN &=\{\mathbf{x\in S}(g)|d\in\mathbf{x},e\notin\mathbf{x}\},\\
NI &=\{\mathbf{x\in S}(g)|d\notin\mathbf{x},e\in\mathbf{x}\},\\
NN &=\{\mathbf{x\in S}(g)|d,e\notin\mathbf{x}\}.\\
\end{align*}
Because of the markings representing two unknots,  using the spans of them, we have the decomposition of the chain complex $N_0\cong \widetilde{II}\oplus\widetilde{IN}\oplus\widetilde{NI}\oplus\widetilde{NN}$.
Let $\mathbb{II}$, $\mathbb{IN}$, $\mathbb{NI}$, and $\mathbb{NN}$ be the subcomplexes of $I_1$ isomorphic to $\widetilde{II}$, $\widetilde{IN}$, $\widetilde{NI}$, and $\widetilde{NN}$, respectively.
Then we have the decomposition of the chain complex $I_1\cong \mathbb{II}\oplus\mathbb{IN}\oplus\mathbb{NI}\oplus\mathbb{NN}$.

\begin{rem}
In this case, the chain map $\partial'|_{I_1}\colon I_1\to N_0$ is not an isomorphism.
The isomorphism $I_1\cong N_0$ is given by the chain map counting only one empty rectangle whose northeast corner is $c$.
\end{rem}

\begin{prop}
\label{prop:partial-disjoint-trivial}
The induced map on homology $H(\partial ')\colon H(\widetilde{I}) \to H(\widetilde{N})$ is trivial.   
\end{prop}
\begin{proof}
Let $\zeta$ be a non-zero element of $H(\widetilde{I})$.
We can assume that $\zeta$ is represented by the sum of the states of $I_1$ as $\mathbf{x}_1+\dots+\mathbf{x}_s$.
We will give the element $\eta\in N_0$ such that $\partial_N(\eta)=\partial'(\mathbf{x}_1+\dots+\mathbf{x}_s)$.

For $i=1,\dots,s$, let $\mathbf{x}_i=\begin{pmatrix}
\mathbf{x}_{i_{11}} & \mathbf{x}_{i_{12}} \\
\mathbf{x}_{i_{21}} & \mathbf{x}_{i_{22}} \\
\end{pmatrix}$.
We remark that $\mathbf{x}_{i_{12}}=\{c\}=\{\alpha_{n+3}\cap\beta_{n+3}\}$.

Since $I_1$ is decomposed into four chain complexes, it is sufficient to consider the following four cases.
\begin{case}
$\mathbf{x}_1+\dots+\mathbf{x}_s$ is a cycle of $\mathbb{II}$.
Then $\mathbf{x}_{i_{21}}=\{\alpha_{1}\cap\beta_{1}\}$.
In this case, we have $\partial'(\mathbf{x}_i)=0$ because $\partial'$ counts exactly two empty rectangles whose northeast and southwest corners are $\mathbf{x}_{i_{12}}$ and $\mathbf{x}_{i_{21}}$.
\end{case}

\begin{case}
$\mathbf{x}_1+\dots+\mathbf{x}_s$ is a cycle of $\mathbb{IN}$.
Then $\mathbf{x}_{i_{21}}$ is on the vertical circle $\beta_1$ and $\mathbf{x}_{i_{21}}\neq \{\alpha_1\cap\beta_{1}\}$.
Let $\mathbf{x}_{i_{21}}=\{\alpha_j\cap\beta_1\}$.
In this case, $\mathbf{x}_{i_{22}}$ has a point on $\alpha_1$.
Write it as $\alpha_1\cap\beta_k$.

Consider a linear map $F\colon \mathbb{IN}\to N_1$ whose value on $\mathbf{x}=\begin{pmatrix}
\mathbf{x}_{11} & \mathbf{x}_{12} \\
\mathbf{x}_{21} & \mathbf{x}_{22} \\
\end{pmatrix}\in \mathbb{IN}$ is given by $F(\mathbf{x})=\begin{pmatrix}
\mathbf{y}_{11} & \mathbf{y}_{12} \\
\mathbf{y}_{21} & \mathbf{y}_{22} \\
\end{pmatrix}$, where

\begin{align*}
\mathbf{y}_{11} &= \mathbf{x}_{11},\\
\mathbf{y}_{12} &= \{\alpha_{n+3}\cap\beta_{k}\},\\
\mathbf{y}_{21} &= \{\alpha_{1}\cap\beta_{1}\},\\
\mathbf{y}_{22} &= (\mathbf{x}_{22}\cup\{\alpha_j\cap\beta_{n+3}\})\setminus\{\alpha_1\cap\beta_k\}.
\end{align*}

\begin{clm}
$\partial_N(F(\mathbf{x}_1+\dots+\mathbf{x}_s))=\partial'(\mathbf{x}_1+\dots+\mathbf{x}_s)$.
\end{clm}
\begin{proof}
For each $i$, $\partial'(\mathbf{x}_i)$ is the sum of two states of $\widetilde{N}$.
A direct computation shows that $\partial_N(F(\mathbf{x}_i))$ contains these two states.
Let $\partial_N(F(\mathbf{x}_i))=\partial'(\mathbf{x}_i)+\phi(F(\mathbf{x}_i))$.

\begin{figure}
\includegraphics[scale=0.23]{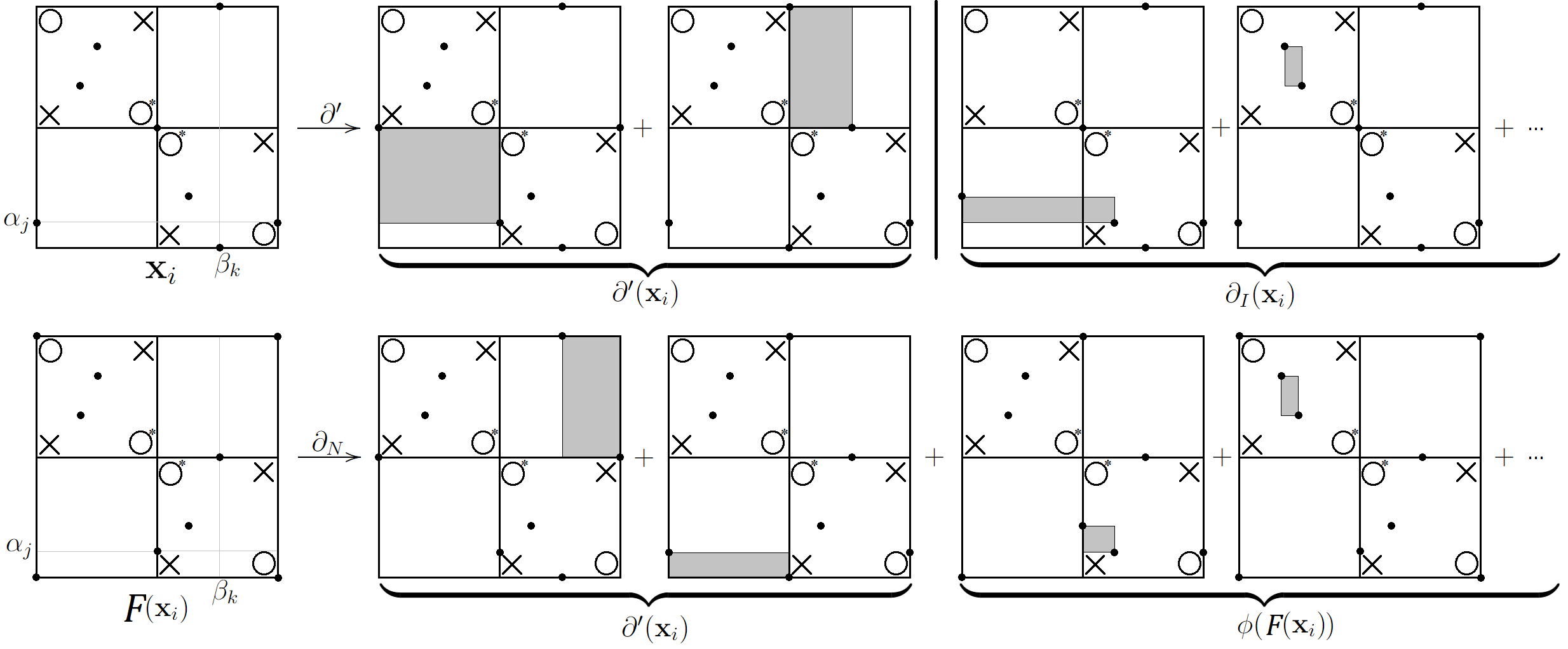}
\caption{\textbf{Case 2.} The correspondence between empty rectangles from $\mathbf{x}_i$ and from $F(\mathbf{x}_i)$.}
\label{fig:dis-case2}
\end{figure}
We will show that 
\begin{align}
\label{eq:case2phiF=Fpartial}
\phi(F(\mathbf{x}_i))=F(\partial_I(\mathbf{x}_i)).
\end{align}
Consider the empty rectangles counted by $\partial_I(\mathbf{x}_i)$ and $\phi(G(\mathbf{x}_i))$.
Let $r\in\mathrm{Rect}^\circ(\mathbf{x}_i,\mathbf{z})$ be an empty rectangle counted by $\partial_I(\mathbf{x}_i)$.
We see that there exist the corresponding empty rectangle $r'\in\mathrm{Rect}^\circ(F(\mathbf{x}_i),\mathbf{z}')$ counted by $\phi(F(\mathbf{x}_i))$.
We have three cases (see Figure \ref{fig:dis-case2}).
\begin{itemize}
\item If $r$ moves the point $\alpha_j\cap\beta_1\in\mathbf{x}_i$, then $r\cap g_{11}= \emptyset$ and $r\cap g_{22}\neq \emptyset$.
Let $r'$ be the rectangle obtained from $r$ by replacing the corner point $\alpha_j\cap\beta_1$ with $\alpha_j\cap \beta_{n+3}$.
Since the two points $\beta_1\cap\mathbf{z}$ and $\beta_{n+3}\cap\mathbf{z}'$ are on the same horizontal circle, we have $F(\mathbf{z})=\mathbf{z}'$.
\item If $r$ moves the point $\alpha_1\cap\beta_k\in\mathbf{x}_i$, then $r\cap g_{11}= \emptyset$ and $r\cap g_{22}\neq \emptyset$.
Let $r'$ be the rectangle obtained from $r$ by replacing the corner point $\alpha_1\cap\beta_k$ with $\alpha_{n+3}\cap \beta_k$.
Since the two points $\alpha_1\cap\mathbf{z}$ and $\alpha_{n+3}\cap\mathbf{z}'$ are on the same vertical circle, we have $F(\mathbf{z})=\mathbf{z}'$.
\item If $r$ preserves $\alpha_j\cap\beta_1, \alpha_1\cap\beta_k\in\mathbf{x}_i$, then $r'=r$.
Clearly we have $F(\mathbf{z})=\mathbf{z}'$.
\end{itemize}
Conversely, for each empty rectangle $r'$ of $\phi(F(\mathbf{x}_i))$, there exists an empty rectangle of $\partial_I(\mathbf{x}_i)$ corresponding to $r'$.
Thus (\ref{eq:case2phiF=Fpartial}) is proved.

Finally, (\ref{eq:case2phiF=Fpartial}) gives
\begin{align*}
\partial_N(F(\mathbf{x}_1+\dots+\mathbf{x}_s))&=\partial'(\mathbf{x}_1+\dots+\mathbf{x}_s)+\phi(F(\mathbf{x}_1))+\dots+\phi(F(\mathbf{x}_s))\\
&=\partial'(\mathbf{x}_1+\dots+\mathbf{x}_s)+F(\partial_I(\mathbf{x}_1))+\dots+F(\partial_I(\mathbf{x}_s))\\
&=\partial'(\mathbf{x}_1+\dots+\mathbf{x}_s)+F(\partial_I(\mathbf{x}_1+\dots+\mathbf{x}_s))\\
&=\partial'(\mathbf{x}_1+\dots+\mathbf{x}_s).
\end{align*}

\end{proof}
\end{case}

\begin{case}
$\mathbf{x}_1+\dots+\mathbf{x}_s$ is a cycle of $\mathbb{NI}$.
The same result as Case 2 is proved by switching $\{\alpha_i\}_{i=1}^{2n+4}$ and $\{\beta_i\}_{i=1}^{2n+4}$.
\end{case}

\begin{figure}
\includegraphics[scale=0.24]{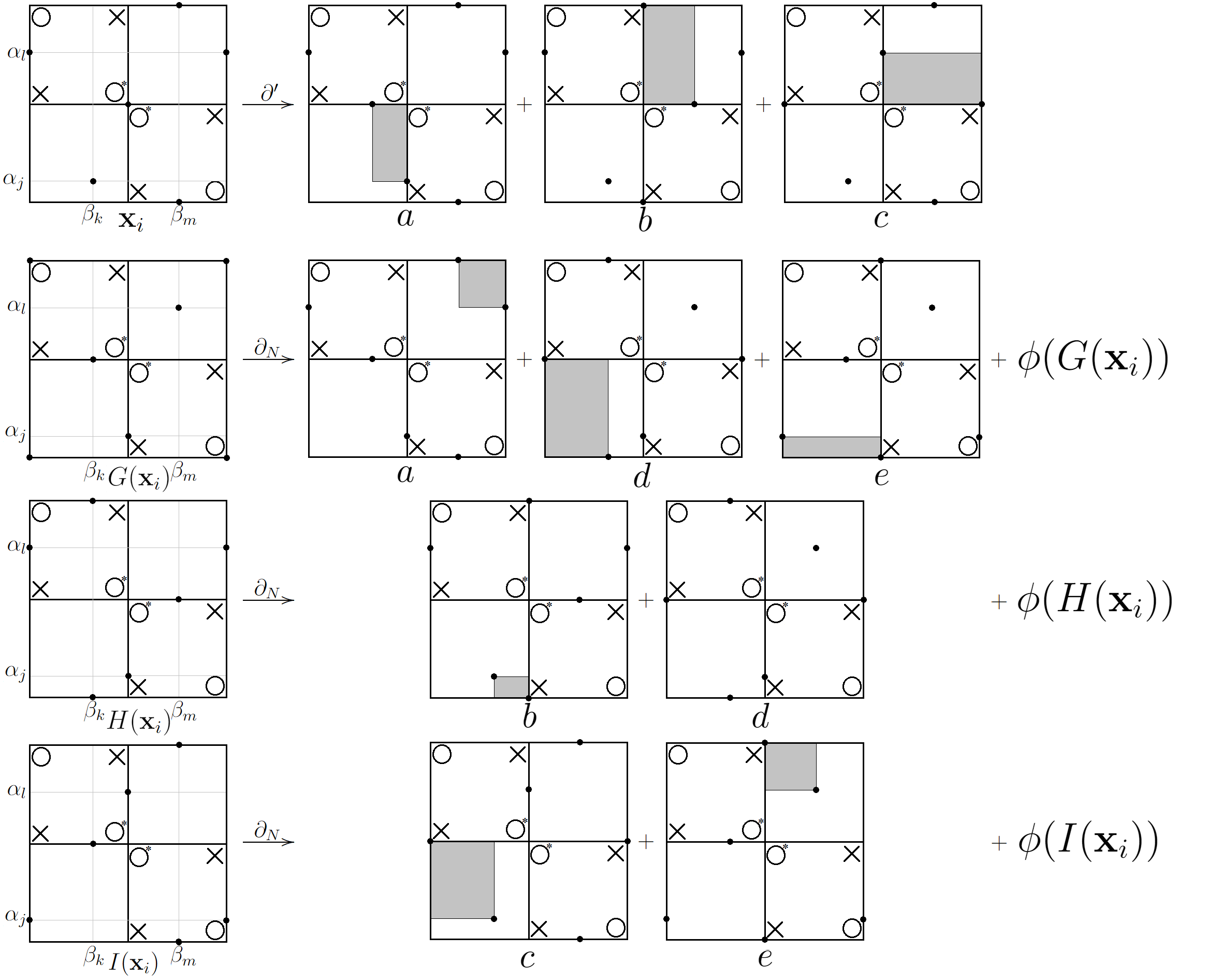}
\caption{\textbf{Case 4.} Here are the empty rectangles from $\mathbf{x}_i$, $\mathbf{y}_i$, $\mathbf{z}_i$, and $\mathbf{w}_i$. This figure implies that $\partial_N(\mathbf{y}_i+\mathbf{z}_i+\mathbf{w}_i)=\partial'(\mathbf{x}_i)+\phi(\mathbf{y}_i)+\phi(\mathbf{z}_i)+\phi(\mathbf{w}_i)$.}
\label{fig:dis-case4}
\end{figure}
\begin{figure}
\includegraphics[scale=0.25]{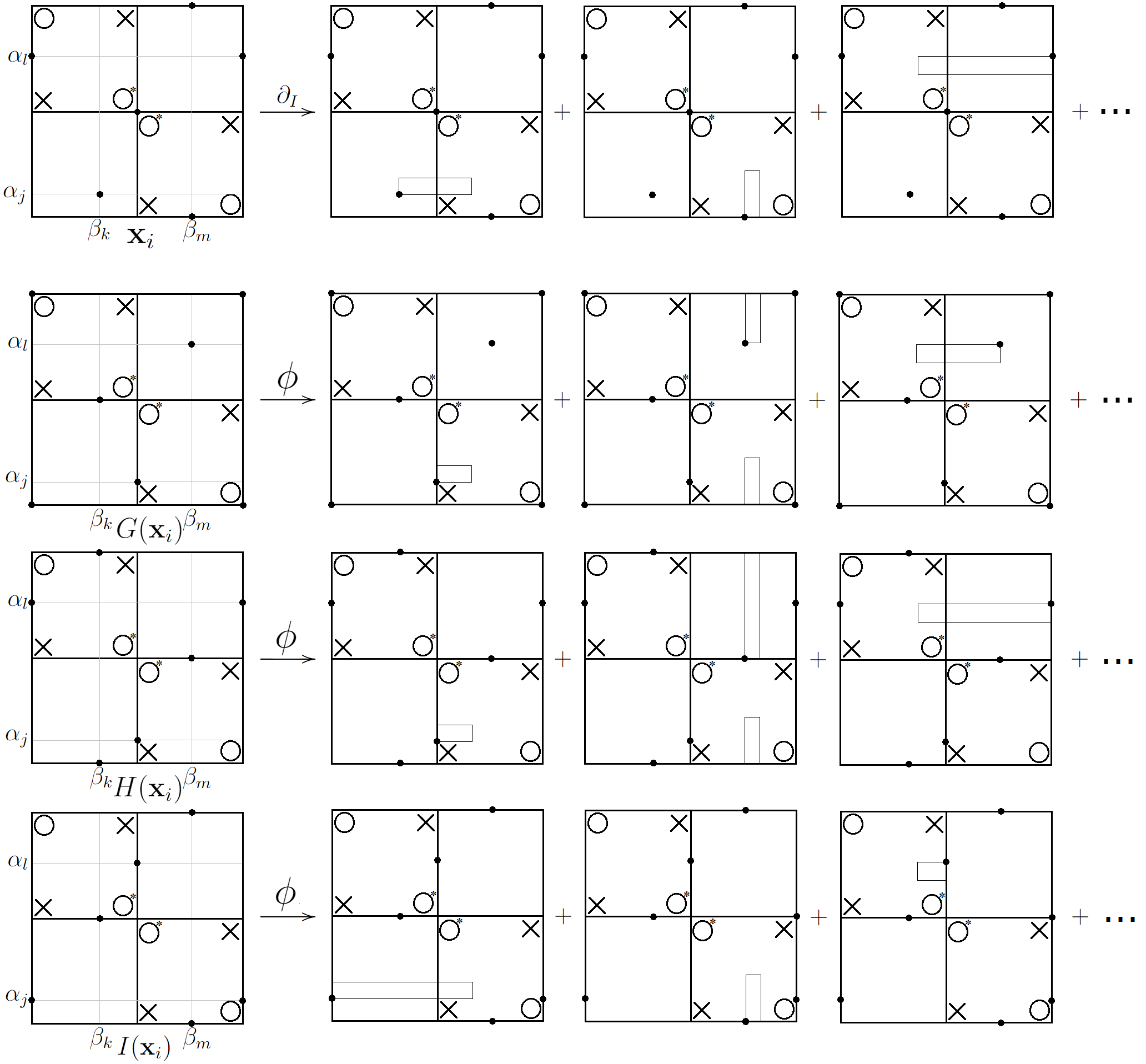}
\caption{\textbf{Case 4.} The rectangles counted by $\partial_I$ and the related rectangles counted by $\phi$.}
\label{fig:dis-case4-check}
\end{figure}

\begin{case}
$\mathbf{x}_1+\dots+\mathbf{x}_s$ is a cycle of $\mathbb{NN}$.
Then the point $\mathbf{x}_{i_{21}}$ is not on $\alpha_1$ or $\beta_1$.
Let $\mathbf{x}_{i_{21}}=\{\alpha_{j}\cap\beta_{k}\}$.
In this case, $\mathbf{x}_{i_{11}}$ has a point on $\beta_1$ and $\mathbf{x}_{i_{22}}$ has a point on $\alpha_1$.
Let $\alpha_l\cap\beta_1\in\mathbf{x}_{i_{11}}$ and $\alpha_1\cap\beta_m\in\mathbf{x}_{i_{22}}$.

Consider three linear maps $G,H,I\colon\mathbb{NN}\to N_1$ whose values on $\mathbf{x}=\begin{pmatrix}
\mathbf{x}_{11} & \mathbf{x}_{12} \\
\mathbf{x}_{21} & \mathbf{x}_{22} \\
\end{pmatrix}\in \mathbb{NN}$ are given by
\[
G(\mathbf{x})=\begin{pmatrix}
\mathbf{y}_{11} & \mathbf{y}_{12} \\
\mathbf{y}_{21} & \mathbf{y}_{22} \\
\end{pmatrix}, 
H(\mathbf{x})=\begin{pmatrix}
\mathbf{z}_{11} & \mathbf{z}_{12} \\
\mathbf{z}_{21} & \mathbf{z}_{22} \\
\end{pmatrix},
I(\mathbf{w})=\begin{pmatrix}
\mathbf{w}_{11} & \mathbf{w}_{12} \\
\mathbf{w}_{21} & \mathbf{w}_{22} \\
\end{pmatrix},
\]
where
\begin{align*}
\mathbf{y}_{11} &= (\mathbf{x}_{11}\cup\{\alpha_{n+3}\cap\beta_{k}\})\setminus\{\alpha_l\cap\beta_1\},\\
\mathbf{y}_{12} &= \{\alpha_{l}\cap\beta_{m}\},\\
\mathbf{y}_{21} &= \{\alpha_{1}\cap\beta_{1}\},\\
\mathbf{y}_{22} &= (\mathbf{x}_{22}\cup\{\alpha_j\cap\beta_{n+3}\})\setminus\{\alpha_1\cap\beta_m\},\\
\end{align*}
\begin{align*}
\mathbf{z}_{11} &= (\mathbf{x}_{11}\cup\{\alpha_{l}\cap\beta_{1}\})\setminus\{\alpha_l\cap\beta_1\},\\
\mathbf{z}_{12} &= \{\alpha_{n+3}\cap\beta_{m}\},\\
\mathbf{z}_{21} &= \{\alpha_{1}\cap\beta_{k}\},\\
\mathbf{z}_{22} &= (\mathbf{x}_{22}\cup\{\alpha_j\cap\beta_{n+3}\})\setminus\{\alpha_1\cap\beta_m\},\\
\end{align*}
\begin{align*}
\mathbf{w}_{11} &= (\mathbf{x}_{11}\cup\{\alpha_{n+3}\cap\beta_{k}\})\setminus\{\alpha_l\cap\beta_1\},\\
\mathbf{w}_{12} &= \{\alpha_{l}\cap\beta_{n+3}\},\\
\mathbf{w}_{21} &= \{\alpha_{j}\cap\beta_{1}\},\\
\mathbf{w}_{22} &= (\mathbf{x}_{22}\cup\{\alpha_1\cap\beta_{m}\})\setminus\{\alpha_1\cap\beta_m\}.
\end{align*}

\begin{clm}
$\partial_N(\sum_{i=1}^s(G(\mathbf{x}_i)+H(\mathbf{x}_i)+I(\mathbf{x}_i))=\partial'(\mathbf{x}_1+\dots+\mathbf{x}_s)$.
\end{clm}
For each $i$, $\partial'(\mathbf{x}_i)$ is the sum of three states of $\widetilde{N}$.

For $G(\mathbf{x}_i)$, exactly three empty rectangles counted by $\partial_N$ have the point $\alpha_1\cap\beta_1$ as their corner.
Let $\partial_N(G(\mathbf{x}_i))=\psi(G(\mathbf{x}_i))+\phi(G(\mathbf{x}_i))$, where $\psi(G(\mathbf{x}_i))$ is the sum of three states obtained by counting these three rectangles and $\phi(G(\mathbf{x}_i))$ is the sum of others.

For $H(\mathbf{x}_i)$, exactly two rectangles counted by $\partial_N$ move two of the four points on $\alpha_1$, $\alpha_{n+3}$, $\beta_1$, and $\beta_{n+3}$.
Write $\partial_N(H(\mathbf{x}_i))=\psi(H(\mathbf{x}_i))+\phi(H(\mathbf{x}_i))$, where $\psi(H(\mathbf{x}_i))$ is the two states obtained by counting these two rectangles and $\phi(H(\mathbf{x}_i))$ is the sum of others.

For $I(\mathbf{x}_i)$, write $\partial_N(I(\mathbf{x}_i))=\psi(I(\mathbf{x}_i))+\phi(I(\mathbf{x}_i))$ in the same way as $H(\mathbf{x}_i)$.

A direct computation (see Figure \ref{fig:dis-case4}) shows
\begin{equation}
\label{eq:partial'-GHI}
\partial'(\mathbf{x}_i)=\psi(G(\mathbf{x}_i))+\psi(H(\mathbf{x}_i))+\psi(I(\mathbf{x}_i)).
\end{equation}
Then, we will prove the following equations:
\begin{align}
\label{eq:case4phiG=Gpartial}
\phi(G(\mathbf{x}_i))&=G(\partial_I(\mathbf{x}_i)),\\
\label{eq:case4phiH=Hpartial}
\phi(H(\mathbf{x}_i))&=H(\partial_I(\mathbf{x}_i)),\\
\label{eq:case4phiI=Ipartial}
\phi(I(\mathbf{x}_i))&=I(\partial_I(\mathbf{x}_i)).
\end{align}

Consider the empty rectangles counted by $\partial_I(\mathbf{x}_i)$ and $\phi(G(\mathbf{x}_i))$.
Let $r$ be an empty rectangle counted by $\partial_I(\mathbf{x}_i)$.
Suppose that $r\in\mathrm{Rect}^\circ(\mathbf{x}_i,\mathbf{z})$.
There exists the corresponding empty rectangle $r'\in\mathrm{Rect}^\circ(G(\mathbf{x}_i),\mathbf{z}')$ counted by $\phi(G(\mathbf{x}_i))$.
We have five cases (see Figure \ref{fig:dis-case4-check}).
\begin{itemize}
\item If $r$ moves the point $\alpha_j\cap\beta_k\in\mathbf{x}_i$ and $r\cap g_{11}\neq\emptyset$, let $r'$ be the rectangle obtained from $r$ by replacing the corner point $\alpha_j\cap\beta_k$ with $\alpha_{n+3}\cap \beta_{k}$.
Since the two points $\alpha_{j}\cap\mathbf{z}$ and $\alpha_{n+3}\cap\mathbf{z}'$ are on the same vertical circle, we have $G(\mathbf{z})=\mathbf{z}'$.
\item If $r$ moves the point $\alpha_j\cap\beta_k\in\mathbf{x}_i$ and $r\cap g_{22}\neq\emptyset$, let $r'$ be the rectangle obtained from $r$ by replacing the corner point $\alpha_j\cap\beta_k$ with $\alpha_{j}\cap \beta_{n+3}$.
Since the two points $\beta_k\cap\mathbf{z}$ and $\beta_{n+3}\cap\mathbf{z}'$ are on the same horizontal circle, we have $G(\mathbf{z})=\mathbf{z}'$.
\item If $r$ moves the point $\alpha_1\cap\beta_m\in\mathbf{x}_i$, then $r\cap g_{11}= \emptyset$ and $r\cap g_{22}\neq \emptyset$.
Let $r'$ be the rectangle obtained from $r$ by replacing the corner point $\alpha_1\cap\beta_m$ with $\alpha_{l}\cap \beta_m$.
Since the two points $\alpha_1\cap\mathbf{z}$ and $\alpha_{l}\cap\mathbf{z}'$ are on the same vertical circle, we have $G(\mathbf{z})=\mathbf{z}'$.
\item If $r$ moves the point $\alpha_l\cap\beta_1\in\mathbf{x}_i$, then $r\cap g_{11}\neq \emptyset$ and $r\cap g_{22}=\emptyset$.
Let $r'$ be the rectangle obtained from $r$ by replacing the corner point $\alpha_l\cap\beta_1$ with $\alpha_{l}\cap \beta_m$.
Since the two points $\beta_1\cap\mathbf{z}$ and $\beta_{m}\cap\mathbf{z}'$ are on the same horizontal circle, we have $G(\mathbf{z})=\mathbf{z}'$.
\item If $r$ preserves $\alpha_j\cap\beta_1, \alpha_1\cap\beta_k\in\mathbf{x}_i$, then let $r'=r$.
We have $G(\mathbf{z})=\mathbf{z}'$.
\end{itemize}
Conversely, for each empty rectangle $r'$ of $\phi(G(\mathbf{x}_i))$, there exists an empty rectangle of $\partial_I(\mathbf{x}_i)$ corresponding to $r'$.
Thus (\ref{eq:case4phiG=Gpartial}) is proved.

Consider the empty rectangles counted by $\partial_I(\mathbf{x}_i)$ and $\phi(H(\mathbf{x}_i))$.
Let $r$ be an empty rectangle appearing in $\partial_I(\mathbf{x}_i)$.
Suppose that $r\in\mathrm{Rect}^\circ(\mathbf{x}_i,\mathbf{z})$.
There exists the corresponding empty rectangle $r'\in\mathrm{Rect}^\circ(H(\mathbf{x}_i),\mathbf{z}')$ appearing in $\phi(H(\mathbf{x}_i))$ as follows.
We have five cases.
\begin{itemize}
\item If $r$ moves the point $\alpha_j\cap\beta_k\in\mathbf{x}_i$ and $r\cap g_{11}\neq\emptyset$, let $r'$ be the rectangle obtained from $r$ by replacing the corner point $\alpha_j\cap\beta_k$ with $\alpha_{1}\cap \beta_{k}$.
Since the two points $\alpha_{j}\cap\mathbf{z}$ and $\alpha_{1}\cap\mathbf{z}'$ are on the same vertical circle, we have $F(\mathbf{z})=\mathbf{z}'$.
\item If $r$ moves the point $\alpha_j\cap\beta_k\in\mathbf{x}_i$ and $r\cap g_{22}\neq\emptyset$, let $r'$ be the rectangle obtained from $r$ by replacing the corner point $\alpha_j\cap\beta_k$ with $\alpha_{j}\cap \beta_{n+3}$.
Since the two points $\beta_k\cap\mathbf{z}$ and $\beta_{n+3}\cap\mathbf{z}'$ are on the same horizontal circle, we have $H(\mathbf{z})=\mathbf{z}'$.
\item If $r$ moves the point $\alpha_1\cap\beta_m\in\mathbf{x}_i$, then $r\cap g_{11}= \emptyset$ and $r\cap g_{22}\neq \emptyset$.
Let $r'$ be the rectangle obtained from $r$ by replacing the corner point $\alpha_1\cap\beta_m$ with $\alpha_{n+3}\cap \beta_m$.
Since the two points $\alpha_1\cap\mathbf{z}$ and $\alpha_{n+3}\cap\mathbf{z}'$ are on the same vertical circle, we have $H(\mathbf{z})=\mathbf{z}'$.
\item If $r$ moves the point $\alpha_l\cap\beta_1\in\mathbf{x}_i$, then $r\cap g_{11}\neq \emptyset$ and $r\cap g_{22}=\emptyset$.
Let $r'=r$ and we have $H(\mathbf{z})=\mathbf{z}'$.
\item If $r$ preserves $\alpha_j\cap\beta_1, \alpha_1\cap\beta_k\in\mathbf{x}_i$, then $r'=r$.
We have $H(\mathbf{z})=\mathbf{z}'$.
\end{itemize}
Conversely, for each empty rectangle $r'$ of $\phi(H(\mathbf{x}_i))$, there exists an empty rectangle of $\partial_I(\mathbf{x}_i)$ corresponding to $r'$.
Thus (\ref{eq:case4phiH=Hpartial}) is proved.

(\ref{eq:case4phiI=Ipartial}) can be proved in the same way as (\ref{eq:case4phiH=Hpartial}).

Finally, (\ref{eq:partial'-GHI})-(\ref{eq:case4phiI=Ipartial}) give
\begin{align*}
&\partial_N(\sum_{i=1}^s(G(\mathbf{x}_i)+H(\mathbf{x}_i)+I(\mathbf{x}_i))\\
&=\partial'(\sum_{i=1}^s\mathbf{x}_i)+\sum_{i=1}^s\phi(G(\mathbf{x}_1))+\sum_{i=1}^s\phi(H(\mathbf{x}_1))+\sum_{i=1}^s\phi(I(\mathbf{x}_1))\\
&=\partial'(\sum_{i=1}^s\mathbf{x}_i)+G(\partial_I(\sum_{i=1}^s\mathbf{x}_i))+H(\partial_I(\sum_{i=1}^s\mathbf{x}_i))+I(\partial_I(\sum_{i=1}^s\mathbf{x}_i))\\
&=\partial'(\sum_{i=1}^s\mathbf{x}_i).
\end{align*}
\end{case}
\end{proof}

\begin{proof}[proof of Theorem \ref{thm:disjoint}]
Using Lemma \ref{lem:adding-unknot}, we have 
\[
\widehat{HF}(g',\omega')\cong \widehat{HF}(g_\sqcup,\omega_{\sqcup})\otimes W(0)^{\otimes 2}.
\]
Proposition \ref{prop:tilde-hat} gives
\[
\widetilde{HF}(g',\omega')\cong \widetilde{HF}(g_\sqcup,\omega_\sqcup)\otimes W(0)^{\otimes 2} \otimes W(1)^{\otimes 2}.
\]
Let $g_1\sqcup\mathcal{O}$ (respectively $g_2\sqcup\mathcal{O}$) be a $(n+2)\times (n+2)$ graph grid diagram which is the same as the upper left (respectively lower right) $(n+2)\times (n+2)$ block of $g'$.
Let $\omega_1\sqcup\mathcal{O}$ and $\omega_2\sqcup\mathcal{O}$ be weights for $g_1\sqcup\mathcal{O}$ and $g_2\sqcup\mathcal{O}$ respectively naturally induced by $\omega'$.
Then Lemmas \ref{lem:N0-CFxCF} and \ref{prop:partial-disjoint-trivial} imply that 
\[
\widetilde{HF}(g',\omega')\cong \widetilde{HF}(g_1\sqcup\mathcal{O},\omega_1\sqcup\mathcal{O})\otimes \widetilde{HF}(g_1\sqcup\mathcal{O},\omega_1\sqcup\mathcal{O}) \otimes W(0).
\]
Proposition \ref{prop:tilde-hat} and Lemma \ref{lem:adding-unknot} give 
\[
\widetilde{HF}(g_i\sqcup\mathcal{O},\omega_i\sqcup\mathcal{O})\cong \widetilde{HF}(g_1,\omega_1)\otimes W(0) \otimes W(1),
\]
for $i=1,2$.
Combining these equations, we have
\[
\widetilde{HF}(g_\sqcup,\omega_\sqcup)\otimes W(0)^{\otimes 2} \otimes W(1)^{\otimes 2} \cong \widetilde{HF}(g_1,\omega_1)\otimes \widetilde{HF}(g_1,\omega_1) \otimes W(0)^{\otimes 3}\otimes W(1)^{\otimes 2}.
\]
and hence we obtain
\[
\widetilde{HF}(g_\sqcup,\omega_\sqcup) \cong \widetilde{HF}(g_1,\omega_1)\otimes \widetilde{HF}(g_1,\omega_1) \otimes W(0).
\]
Finally, Proposition \ref{prop:tilde-hat} gives
\[
\widehat{HF}(g_\sqcup,\omega_\sqcup) \cong \widehat{HF}(g_1,\omega_1)\otimes \widehat{HF}(g_1,\omega_1) \otimes W(0).
\]
\end{proof}

\section{An application and examples}
Let $G$ be an abstract graph.
$G$ is \textbf{planar} if there is an embedding of $G$ into $\mathbb{R}^2$.
For a planar graph $G$, a spatial graph $f(G)$ is \textbf{trivial} if $f(G)$ is ambient isotopic to an embedding of $G$ into $\mathbb{R}^2\subset\mathbb{R}^3$.
It is known that a trivial spatial embedding of a planar graph is unique up to ambient isotopy in $\mathbb{R}^3$ \cite{Homeomorphic-continuous-curves-in-2-space-are-isotopic-in-3-space}.

The grid homology gives obstructions to the trivial spatial handcuff graph.
\begin{cor}
\label{cor:nontrivial}
Let $f\colon G\to S^3$ be a spatial embedding of a handcuff graph.
Take any balanced coloring $\omega$ for $f$.
If $\widehat{HF}(f,\omega)$ is nontrivial, then $f(G)$ is nontrivial.
\end{cor}
\begin{proof}
A handcuff graph is clearly planar.
The trivial embedding of it has a cut edge.
Then Theorem \ref{thm:cutedge} (2) completes the proof.
\end{proof}
\begin{rem}
\begin{itemize}
    \item This corollary holds for any planar graph with a cut edge.
    \item The converse of this corollary does not hold.
For example, the grid homology of the spatial graph of the center of Figure \ref{fig:3handcuff} is also trivial.
\end{itemize}
\end{rem}

\begin{figure}
\includegraphics[scale=0.45]{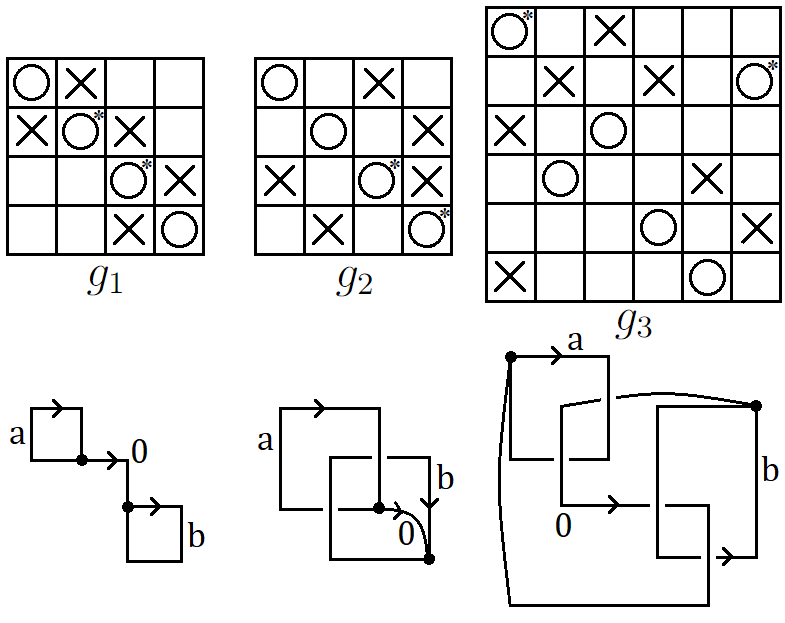}
\caption{Three graph grid diagrams for spatial handcuff graphs.}
\label{fig:3grid-handcuff}
\end{figure}

Here are some computations of grid homology for spatial handcuff graphs. 
Let $g_1$, $g_2$, and $g_3$ be three graph grid diagrams as in Figure \ref{fig:3grid-handcuff}.
Suppose that their balanced colorings, denoted by $\omega_i$ $(i=1,2,3)$, take $a$ and $b$ for two loops and zero for the edge connecting two different vertices.
Direct computations show that
\begin{align*}
\widehat{HF}(g_1,\omega_1)=&0,\\
\widehat{HF}(g_2,\omega_2)=&\mathbb{F}_{(0,a+b)}\oplus\mathbb{F}_{(-1,a)}\oplus\mathbb{F}_{(-1,b)}\oplus\mathbb{F}_{(-2,0)},\\
\widehat{HF}(g_3,\omega_1)=&\mathbb{F}_{(1,a+b)}\oplus\mathbb{F}_{(0,a+b)}\oplus\mathbb{F}_{(0,a)}\oplus\mathbb{F}_{(0,b)}\\
&\oplus\mathbb{F}_{(-1,a)}\oplus\mathbb{F}_{(-1,b)}\oplus\mathbb{F}_{(-1,0)}\oplus\mathbb{F}_{(-2,0)},
\end{align*}
where relative Alexander grading is shifted for simplicity.
Using Corollary \ref{cor:nontrivial}, we see that spatial graphs represented by $g_2$ and $g_3$ are nontrivial.
We remark that when $a=b=1$, $\widehat{HF}(g_2,\omega_2)$ coincides with the grid homology of the positive Hopf link.

\section{Acknowledgement}
I would like to express my sincere gratitude to my supervisor, Tetsuya Ito, for helpful discussions and corrections.
This work was supported by JST, the establishment of university fellowships towards
the creation of science technology innovation, Grant Number JPMJFS2123.

\bibliography{grid}
\bibliographystyle{amsplain} 

\end{document}